\documentclass{amsart}
\usepackage{amsmath, amssymb, amsthm}
\usepackage{mathtools}
\usepackage{amscd}
\usepackage{graphicx} 
\usepackage{enumerate}
\usepackage{mathrsfs}
\usepackage{todonotes}

\def\PP{{\mathbb P}}

\def\v{{\nu}}

\def\Qbar{\overline{\mathbb Q}}

\def\0{{\mathbf 0}}
\def\1{{\mathbf 1}}

\def\cA{{\mathcal A}}
\def\cB{{\mathcal B}}

\def\cF{{\mathcal F}}

\def\cJ{{\mathcal J}}

\def\cO{{\mathcal O}}

\def\cR{{\mathcal R}}

\def\Kbar{{\bar K}}

\def\Gal{\mathrm{Gal}}

\def\supp{\mathrm{supp}}

\usepackage[all]{xy}

\newtheorem{thm}{Theorem}

\newtheorem{prop}[thm]{Proposition}
\newtheorem{lemma}[thm]{Lemma}

\newtheorem{cor}[thm]{Corollary}

\newtheorem*{thm*}{Theorem}
\newtheorem*{alg*}{Algorithm}
\newtheorem*{lemma*}{Lemma}

\theoremstyle{remark}
\newtheorem{rmk}[thm]{Remark}
\newtheorem*{rmk*}{Remark}

\newtheorem*{notation*}{Notation}
\newtheorem{example}[thm]{Example}
\newtheorem*{example*}{Example}

\theoremstyle{definition}
\newtheorem{defn}[thm]{Definition}
\newtheorem*{defn*}{Definition}

\numberwithin{thm}{section}

\newcommand{\mybf}{\mathbb}

\newcommand{\bE}{\mybf{E}}

\newcommand{\bP}{\mybf{P}}
\newcommand{\bR}{\mybf{R}}

\newcommand{\bC}{\mybf{C}}
\newcommand{\bN}{\mybf{N}}

\newcommand{\bQ}{\mybf{Q}}

\newcommand{\bA}{\mybf{A}}

\newcommand{\Q}{\mybf{Q}}

\newcommand{\C}{\mybf{C}}

\renewcommand{\P}{\mybf{P}}

\newcommand{\al}{\alpha}

\providecommand{\abs}[1]{\lvert#1\rvert}
\providecommand{\norm}[1]{\lVert#1\rVert}

\newcommand{\ON}[1]{\operatorname{#1}}

\newcommand{\ra}{\rightarrow}

\newcommand{\twid}[1]{\widetilde{#1}}
\newcommand{\ep}{\epsilon}

\newcommand{\p}{\partial}

\newcommand{\Diag}{\mathrm{Diag}}

\def\talltareesidedbox#1{\setbox0=\hbox{$#1$}\dimen0=\wd0 \advance\dimen0 by3pt\rlap{\hbox{\vrule height10pt width.4pt
 depth2pt \kern-.4pt\vrule height10.4pt width\dimen0 depth-10pt\kern-.4pt \vrule height10pt width.4pt depth2pt}}
 \relax \hbox to\dimen0{\hss$#1$\hss}}
\def\tareesidedbox#1{\setbox0=\hbox{$#1$}\dimen0=\wd0 \advance\dimen0 by3pt\rlap{\hbox{\vrule height8pt width.4pt
 depth2pt \kern-.4pt\vrule height8.4pt width\dimen0 depth-8pt\kern-.4pt \vrule height8pt width.4pt depth2pt}}
\relax \hbox to\dimen0{\hss$#1$\hss}}
\def\shorttareesidedbox#1{\setbox0=\hbox{$#1$}\dimen0=\wd0 \advance\dimen0 by3pt\rlap{\hbox{\vrule height7pt width.4pt
 depth2pt \kern-.4pt\vrule height7.4pt width\dimen0 depth-7pt\kern-.4pt \vrule height7pt width.4pt depth2pt}}
 \relax \hbox to\dimen0{\hss$#1$\hss}}

\newcommand{\N}{\operatorname{N}}

\newcommand{\sP}{\mathsf{P}}
\newcommand{\sA}{\mathsf{A}}

\title[Equidistribution]{Stochastic Equidistribution and Generalized Adelic Measures}

\author[Doyle]{John R. Doyle}
\address{Department of Mathematics\\ Oklahoma State University, Stillwater, OK
74078}
\email{john.r.doyle@okstate.edu}

\author[Fili]{Paul Fili}
\address{Department of Mathematics\\ Oklahoma State University, Stillwater, OK
74078}
\email{paul.fili@okstate.edu}

\author[Tobin]{Bella Tobin}
\address{Department of Mathematics\\ Oklahoma State University, Stillwater, OK
74078}
\email{bella.tobin@okstate.edu}

\date{\today}

\begin{document}

\begin{abstract}
 We study the dynamics of stochastic families of rational maps on the projective line. As such families can be infinite and may not typically be defined over a single number field, we introduce the concept of generalized adelic measures, generalizing previous notions introduced by Favre and Rivera-Letelier and Mavraki and Ye. Generalized adelic measures are defined over the measure space of places of an algebraic closure of the rationals, using a framework established by Allcock and Vaaler. This turns our heights from sums over places into integrals. We prove an equidistribution result for generalized adelic measures, and use this result to prove an equidistribution result for random backwards orbits in stochastic arithmetic dynamics.
\end{abstract}

\maketitle
\tableofcontents

\allowdisplaybreaks[2]

\section{Introduction}

Let $S$ be a countable set of rational maps defined over $\Qbar$. If we endow $S$ with a probability measure $\nu_1$, we think of the pair $(S,\v_1)$ as a \emph{stochastic dynamical system}, where for a map $\varphi \in S$, the quantity $\v_1(\varphi)$ represents the probability of applying $\varphi$ in a random walk. Recently, Healey and Hindes \cite{HealeyHindes} defined a stochastic height function and proved that, under certain assumptions, it exists and defines a Weil height.\footnote{In fact, they proved their result for endomorphisms of projective varieties. However, in this paper, we will restrict our attention to families of rational maps on the projective line.} (See also earlier work and an equidistribution theorem for random sequences of morphisms defined over a single number field in Kawaguchi \cite{KawaguchiRandomIter}.) We briefly recall their construction: We associate to an element $\gamma_n=(\varphi_1,\ldots, \varphi_n)\in S^n$ its natural composition map, and think of it as a function on the projective line $\bP^1(\Qbar)$, that is, we will write $\gamma_n(\al)=\varphi_n\circ \cdots\circ \varphi_1(\al)$ and take the degree of $\gamma_n$ to be $\deg(\gamma_n) = \deg(\varphi_1)\cdots \deg(\varphi_n)$. We endow $S^n$ with the natural product measure $\v_n = \v_1\times \cdots \times \v_1$, so that $\v_n((\varphi_1,\ldots, \varphi_n)) = \v_1(\varphi_1)\cdots \v_1(\varphi_n)$. Then the \emph{stochastic height} associated to $(S,\v_1)$, when it exists, is the function $h_S : \PP^1(\Qbar)\ra [0,\infty)$ defined by
\[
 h_S(\al) = \lim_{n\ra\infty} \bE_{S^n} \frac{1}{\deg(\gamma_n)} h(\gamma_n(\al)),
\]
where $h$ is the absolute logarithmic Weil height on $\P^1(\Q)$, and where the expectation is taken over $\gamma_n\in S^n$ with the probability measure $\v_n$. Notice that if $S=\{\varphi\}$ where $\varphi$ is rational map of degree at least $2$, and $\v_1(\varphi)=1$, this reduces to the usual (dynamical) canonical height associated to $\varphi$.

It is well-known in arithmetic dynamics that points of low height play a special role: the points of height zero are precisely the preperiodic points of $\varphi$, and it is known (see for example \cite{FRL,BakerRumely,ChambertLoirEqui,ChambertLoirThuillierMahler,ThuillierThesis,Yuan}) that sequences of points of small height equidistribute according to certain canonical measures determined by $\varphi$ at every place of the number field over which $\varphi$ is defined. A typical example of this is the sequence of probability measures supported equally on the set of preimages $\varphi^{-n}(\al)$ for each $n\in\bN$, where $\al$ is a point which is not exceptional for the map $\varphi$. These preimage sets consist of points of geometrically decreasing height, and equidistribute according the canonical measure at every place. This canonical measure is naturally supported on the Julia set of the map and thus is also tied to the dynamics of the map.

However, the situation for stochastic dynamical systems is a little bit more complicated: Healey and Hindes prove that while having stochastic height zero does guarantee that the forward orbit under the entire system $S$ is finite, they also prove that there are only finitely many points of height zero, unless the maps all share the same set of preperiodic points, or equivalently, the same Julia set at every place; see \cite[Corollaries 1.4-5]{HealeyHindes}. (We note in passing that typically two maps that have the same Julia sets share a common iterate, though there exceptions; see \cite{YeRationalFunctions}.) Of course, with only finitely many points of height zero, it is not immediately clear that there is any equidistribution theorem to prove. 

The goal of this paper is to prove that, in fact, a natural stochastic analogue of the equidistribution of preimages is true under certain natural restrictions on the stochastic family of rational maps $(S,\v_1)$. Our main theorem, proven as Theorem \ref{thm:stoch-equi} below, is the following:

\begin{thm*}[Equidistribution of random backwards orbits]
Let $S$ be a countable set of rational maps defined over an algebraic closure $\Qbar$, with each map being of degree at least $2$, and let $\nu_1$ be a probability measure on $S$ with respect to which the maps in $S$ are $L^1$ height controlled. Then for almost every place $y$ of $\Qbar$ there exists a certain canonical stochastic dynamical measure $\rho_y$ such that the following is true: Let $\al\in \bP^1(\Qbar)$ be any point which is not in the exceptional set of the stochastic system $S$. Then the backwards orbit measures $\Delta_{n,\al}$ under $S$ converge weakly to $\rho_y$ as $n \to \infty$.
\end{thm*}
The precise definitions used above can be found in Section \ref{sec:stoch-dyn} below, but here we give an informal summary of the conditions of the main theorem:
\begin{enumerate}
    \item \label{item:2} For each rational map, it is known that the canonical height $h_\varphi$ and the standard height differ by at most an explicit constant $C_\varphi$ depending on $\varphi$. Then expected value of this constant over the stochastic family is assumed to be finite.
    \item The backwards orbit measures $\Delta_{n,\alpha}$ are defined recursively by writing each measure $\Delta_{n,\alpha}$ as a countable sum of weighted point masses in $\P^1(\Qbar)$ (starting with $\Delta_{0,\alpha} = \delta_\alpha$, the unit point mass at $\alpha$), then taking all possible preimages of the points in the support of $\Delta_{n,\alpha}$ under the maps in $S$, weighted according to the probability $\nu_1(\varphi)$.
    
    The assumption that $\alpha$ is not exceptional simply means that the grand orbit of $\alpha$ is infinite.
\end{enumerate}

In order to prove (or even precisely state) our theorem, we need to extend the current framework for talking about heights on the projective line. In order to explain what is new in our framework, let us briefly recall the concept of adelic measures due to Favre and Rivera-Letelier \cite{FRL} and its relation to height. 

The notion of an \emph{adelic measure} was introduced by Favre and Rivera-Letelier \cite{FRL} in order to generalize the idea of the canonical Weil height associated to the dynamics of rational map, based on its canonical measure at every place. Adelic measures are families of measures, defined over a number field, which are the `standard measure' (associated to the local Weil height) at almost all places of the base field and which differ from the standard measure in a controlled fashion at the remaining places. 

This notion was later generalized by Mavraki and Ye \cite{MavrakiYe} to the notion of a \emph{quasi-adelic measure}, which allowed infinitely many `bad' places where the measure was not the standard measure, provided that a global summability condition was met in order to ensure that the resulting height was still a Weil height. Mavraki and Ye then proved an equidistribution result for quasi-adelic measures. Their generalization was inspired by the study of dynamical heights in parameterized families, where it turned out that many of the resulting height functions failed to be adelic in the original sense of Favre and Rivera-Letelier. 

The heights associated to adelic and quasi-adelic measures have one common restriction: they are all defined over a single base global field which satisfies the product formula, for example, a number field. Heights are then computed as averages of Galois orbits of points over this base field. As we wish to study the stochastic dynamics of families of maps which may not be defined over a single base field, we will need a more general framework. The framework for this more general definition of the height was developed by Allcock and Vaaler \cite{AV} and will allow us to define the notion of a \emph{generalized adelic measure} below.

As we are interested in heights associated to stochastic dynamical systems, we will give a construction that, under the assumptions on the family $S$ of rational maps with probability measure $\v_1$ given above, allows us to define a generalized adelic measure $\rho$ associated to $S$ such that $h_S = h_\rho$. We note that there are three basic levels of definition for the generalized adelic measure $\rho$:
\begin{enumerate}
    \item If $S$ is finite, then $\rho$ is an adelic measure in the sense of Favre and Rivera-Letelier. This is because there is a single number field over which all maps in $S$ are defined, and a finite number of places of bad reduction in total, so the resulting measure for the stochastic dynamical system will also be defined over a single number field and equal to the standard measure at all but finitely many places.
    \item If $S$ is infinite, but all $\varphi\in S$ are defined over a single number field $K$, then $\rho$ may not be an adelic measure, as there may be infinitely many places of bad reduction. Our $L^1$ height control assumption guarantees, however, that the resulting measure is still a quasi-adelic measure in the sense of Mavraki and Ye.
    \item If $S$ is infinite and the field of definition is of infinite degree, then $\rho$ may not be an adelic measure or quasi-adelic measure, but the $L^1$ height control assumption guarantees that it is a generalized adelic measure.
\end{enumerate}

In Theorem \ref{thm:gen-adelic-equi}, we prove analogues of the main theorems of Favre and Rivera-Letelier \cite[Th\'eor\`emes 1-2]{FRL}, namely, that the heights associated to generalized adelic measures are essentially-positive Weil heights, and that points of small height equidistribute according to the generalized adelic measure; however, this equidistribution result only holds outside a set of places of measure zero, according to a natural measure on the space of places. This measure is the natural extension of the notion of `local degree over global degree' at a place of a number field, and in particular, the set of places of $\Qbar$ which lie over a given place of a number field always has positive measure, so that the equidistribution theorem, when applied to a generalized adelic measure defined over a single number field, applies at every place. 

Proving results in this context has required several innovative definitions and proofs. For example, heights are no longer defined in terms of average Galois orbits of a point over the base field, and some arguments that relied on this concept have changed. We define our heights more generally for discrete probability measures on $\bP^1(\Qbar)$, possibly with infinite support, and we classify which discrete probability measures have finite height.
In fact, we prove that finiteness of the height of $\Delta$ depends only on $\Delta$ and not on which generalized adelic measure is being used (Corollary \ref{cor:finite-height-is-same-for-all-heights}).

The structure of this paper is as follows. In Section \ref{sec:gen-adelic} we define the notion of a generalized adelic measure, introduce the notion of heights for discrete probability measures on $\bP^1(\Qbar)$ and introduce the notion for finite height for a discrete probability measure. 

In Section \ref{sec:equi-for-gen-adelic}, we prove the main theorems on the existence of Weil heights associated to generalized adelic measures. We then introduce a notion of what it means for a sequence of discrete probability measures to be \emph{well-distributed} in Definition \ref{defn:well-distributed}. This notion replaces the classical notion of the degree of our points tending to infinity in the equidistribution theorem, which may no longer apply as even one step in a backwards orbit for an infinite family of maps may result in a discrete probability measure concentrated on a countably infinite set. We then prove a local equidistribution theorem for such measures in Proposition \ref{prop:local-equi}. The global theorem then follows in Theorem \ref{thm:gen-adelic-equi}, the main result of this section. We note that even when our generalized adelic measure is an adelic measure in the original sense of Favre and Rivera-Letelier, these equidistribution theorems are still an expansion of the original equidistribution theorems, as we prove equidistribution for a much wider class of probability measures on $\bP^1(\Qbar)$. 

Finally, in Section \ref{sec:stoch-dyn}, we prove that for an $L^1$ height controlled stochastic family of rational maps, the stochastic height is given by a generalized stochastic measure associated to the family. We prove a stochastic pullback formula for the stochastic dynamical system akin to the classical pullback formula in Theorem \ref{thm:pullback-formula}. We note that an analogue of the pushforward formula does not, in general, hold for stochastic families. Finally, we define the notion of a \emph{random backwards orbit measure} associated to a starting point $\al\in\bP^1(\Qbar)$, and we prove the stochastic equidistribution for random backwards orbits of non-exceptional points in Theorem \ref{thm:stoch-equi}. We conclude with a definition of the stochastic Julia and Fatou sets associated to the stochastic system. 

In a forthcoming paper \cite{DFT_II}, the authors will prove analogues of classical results about the dynamics of the Julia and Fatou sets for stochastic dynamical system. 

\subsection*{Acknowledgments}
The authors would like to thank Joseph Silverman for helpful suggestions, Xander Faber for suggesting a reference to \cite{GublerMfields}, and Paul Nguyen for computational assistance.

\section{Generalized Adelic Measures}\label{sec:gen-adelic}

\subsection{Definitions}
Let us start by recalling the basic notation and measures associated to the space of places of $\Qbar$ from \cite{AV}. This construction makes $\Qbar$ into an $M$-field in the sense of Gubler \cite{GublerMfields}. We note that this definition works more generally for a global base field in place of $\bQ$; however, we will focus on the case of $\bQ$ as the prime field here. For a number field $K$, let $M_K$ denote its set of places. Notice that the collection of sets $M_K$ as $K$ ranges over finite extensions of $\bQ$ forms a projective system with the natural maps $M_L\ra M_K$ for $L\supset K$ given by restriction of places. The set of places of $\Qbar$ is then the projective limit
\[
 Y = \varprojlim_K M_K,
\]
where $K$ ranges over finite extensions of $\bQ$. If we endow each $M_K$ with the discrete topology, this induces a natural projective topology on $Y$. Notice that each $y\in Y$ has an associated absolute value $\abs{\cdot}_y$, and that for any place $v\in M_K$, we have a natural notion of $y\in Y$ extending $v$ or not, so we can still ask if $y$ lies above $v$, denoted by $y\mid v$ as usual. Further, each $y\mid p$ corresponds to a choice of embedding $\Qbar\hookrightarrow \bC_p$, and by making a choice for one embedding at each prime $p$, the remaining choices are all determined. We shall assume that we have made such a choice for each rational prime $p$.

Let us denote, for each number field $K$ and place $v\in M_K$, the set  
\[
 Y(K,v) = \{ y\in Y : y\mid v\}.
\]
Notice that with the subspace topology induced by $Y$, this set is compact open (in fact, it has the profinite topology, as a projective limit of the finite sets of places above $v$ for each number field containing $K$). Allcock and Vaaler demonstrated \cite[\S 4]{AV} that the absolute Galois group acts naturally on each set $Y(\bQ,p)$ for $p\in M_\bQ$, and that its Haar measure induces a natural Borel probability measure $\mu = \mu|_{Y(\bQ,p)}$ such that for each $v\mid p$ a place of $K/\bQ$, 
\[
 \mu(Y(K,v)) = \frac{[K_v : \bQ_p]}{[K:\bQ]}.
\]
This defines a Borel measure $\mu$ on $Y$ given by using this local measure on every (disjoint) set $Y(\bQ,p)$. For an algebraic number $\al\in K$, the usual absolute logarithmic Weil height can then be expressed in two ways: 
\[
 h(\al) = \sum_{v\in M_K} \frac{[K_v:\bQ_v]}{[K:\bQ]} \log^+ \abs{\al}_v = \int_Y \log^+ \abs{\al}_y\,d\mu(y), 
\]
where the single bar notation $\abs{\,\cdot\,}_y$ for $y\mid p$ agrees with the usual normalization of the absolute value of $\bC_p$. Notice that in the above equation, we are thinking of $f(y) = \log^+ \abs{\al}_y$ as a real-valued function of the places $Y$ of $\Qbar$. Indeed, it is easy to see that $f(y)$ is locally constant on the sets $Y(K,v)$ as $v$ ranges over $M_K$, making it continuous, and further that it is compactly supported, as it is nonzero at only a finite number of places $v$ of $K$. 

Define the \emph{standard measure} at a place $v$ to be
\begin{equation}
 \lambda_v = \begin{cases}
              d\theta/2\pi & \text{if }v\mid \infty,\\
              \delta_{\zeta_{0,1}} & \text{if }v\nmid \infty,
             \end{cases}
\end{equation}
where $d\theta/2\pi$ is the usual normalized arc-length measure of the unit circle in $\sA^1_v = \bC$ if $v\mid\infty$, and $\delta_{\zeta_{0,1}}$ denotes the point mass at the Gauss point $\zeta_{0,1}$ if $\sA^1_v = \bA^1_\text{Berk}(\bC_v)$ if $v\nmid \infty$. Notice that this definition makes sense regardless of whether $v$ is a place of a number field $K$ or a place of $\Qbar$. 

Favre and Rivera-Letelier define an \emph{adelic measure} $\rho$, defined over a base number field $K$, as a collection of Borel probability measures $\rho = (\rho_v)_{v\in M_K}$ where each $\rho_v$ is a Borel probability measure on $\sP^1_v$, the Berkovich projective line over $\bC_v$, with the following conditions:
\begin{enumerate}
 \item At each $v\in M_K$, $\rho_v$ admits a continuous potential with respect to the standard measure, that is, there is a continuous function $g_v : \sP^1_v \ra \bR$ such that $\Delta g_v = \rho_v - \lambda_v$, where $\Delta$ denotes the normalized Laplacian\footnote{When $v$ is archimedean, $\sP^1_v = \bP^1(\bC)$, and $\Delta = (1/2\pi)(\partial^2 / \partial x^2 + \partial^2 / \partial y^2)$. For the definition in the case of a non-archimedean $v$ we refer the reader to \cite[\S 4]{FRL}.} on the Berkovich projective line $\sP^1_v$.
 \item At all but finitely many places $v\in M_K$, we have $\rho_v = \lambda_v$.
\end{enumerate}
We recall that $\sP^1_v = \sP^1(\bC_v)$ is the Berkovich projective line, a Hausdorff, compact, and uniquely path-connected space in which $\bP^1(\bC_v)$ is a dense subspace. For more details on Berkovich space, we refer the reader to, for example,  \cite{FRL,BakerRumelyBook,BakerBerkArticle}.

They then define a height function $h_\rho$ associated to any adelic measure $\rho$, and prove  \cite[Th\'eor\`eme 1]{FRL} that the height function $h_\rho$ satisfies:
\begin{enumerate}
 \item $h_\rho(\al) = h(\al) + O(1)$ where the big-$O$ constant is independent of $\al$ and depends only on $\rho$, and 
 \item $h_\rho$ is essentially nonnegative, in the sense that for any $\ep>0$, the set $\{\al\in \overline{K} : h_\rho(\al) < -\ep \}$ is finite.
\end{enumerate}

We will need to generalize our notion of adelic measure to deal with the fact that, when we define a stochastic height for an infinite family of maps, first, we may no longer have a single number field over which all of the maps are defined, and second, we may have bad reduction at an infinite number of places.

\begin{defn}\label{defn:gen-adelic-measure}
 For each $y\in Y$ a place of $\Qbar$, let $\cB_y$ denote the space of Borel probability measures on the Berkovich projective line $\sP^1_y$ over $\bC_y$. A \emph{generalized adelic measure} $\rho$ is a function $Y \ra \cB_y$ that satisfies the following conditions:
 \begin{enumerate}
  \item \label{enum:gen-adelic-1}For every rational prime $p$, there exists a measurable function
  \begin{equation*}\label{eqn:g-eqn}
   g : Y(\bQ,p) \times \sP^1(\bC_p) \ra [-\infty,\infty]
  \end{equation*}
  such that, for $\mu$-almost all $\Qbar$-places $y\in Y(\bQ,p)$, the function $g_y(z) = g(y,z)$ is a continuous function $\sP^1(\bC_p) \ra \bR$, normalized so that $g_y(\infty) = 0$, which satisfies $\Delta g_y = \rho_y - \lambda_y$. Note that for any given $y\in Y(\bQ,p)$, we identify $\sP^1_y$ with $\sP^1(\bC_p)$, where the embedding is determined by our choice of place $y$.
  \item \label{enum:gen-adelic-2}If $C : Y \ra [0,\infty]$ is the function defined, for the $\mu$-almost everywhere set where $g_y$ defined above is continuous, by 
  \[
  C(y) = \sup_{z\in \sP^1_v} \abs{g_y(z)},
  \]
  and $C(y) = \infty$ on the remaining $\mu$-measure zero set of places $y$, then $C\in L^1(Y)$, that is,
  \[
  \int_Y C(y) \,d\mu(y) < \infty.
  \]
 \end{enumerate}
\end{defn}

It is important to note that a generalized adelic measure may have a nonempty set of places of $\mu$-measure zero where the measure $\rho_y$ does not admit a continuous potential with respect to the standard measure. At these places, equidistribution results may not hold. An example of such a stochastic dynamical system is given in Example \ref{ex:bad-places}.

\begin{defn}
 Let $\rho$ be a generalized adelic measure, and for each rational prime $p$ let $g : Y(\Q,p) \times \sP^1(\C_p) \to [-\infty,\infty]$ be as in Definition~\ref{defn:gen-adelic-measure}. If there exists a number field $K$ such that $g_{y_1} = g_{y_2}$ whenever $y_1$ and $y_2$ lie over the same place of $K$, then we say that $\rho$ is \emph{defined over $K$}.
\end{defn}

As we noted in the introduction, we will show below in Section \ref{sec:comparision-to-MY} that when a generalized adelic measure is defined over a single number field $K$, the concept reduces to that of a quasi-adelic measure or an adelic measure, depending on whether the number of places where the measure is not equal to the standard measure is finite or infinite. Our goal in introducing this generalization is to naturally capture appropriate limits of such measures as the base fields of definition $K$ grow. But first, we will define the height associated to a generalized adelic measure and prove some of its properties. 

First, for any place $y\in Y$ and Borel probability measures $\rho_y,\sigma_y$ on $\sP^1_y$, define (when it exists) the \emph{local energy pairing} to be:
\begin{equation}\label{eqn:local-pairing}
 (\rho_y,\sigma_y)_y = {\iint}_{\sA^1_y\times\sA^1_y\setminus\Diag_y} -\log \abs{z-w}_y\,d\rho_y(z)\,d\sigma_y(w),
\end{equation}
where $\sA^1_y=\sA^1(\bC_y)$ denotes the Berkovich affine line over $\bC_y$ and $\Diag_y = \{ (z,z) : z\in\bC_y\}$. (Note that we are only excluding the \textit{classical}, that is, type I, points of the diagonal.)

One key fact which we will have occasion to use several times in this paper is that if $\rho_y,\sigma_y$ are two probability measures on $\sP^1_y$ for some place $y\in Y$ which admit continuous potentials, then 
\begin{equation}
 (\rho_y-\sigma_y, \rho_y-\sigma_y)_y\geq 0,
\end{equation} 
with equality if and only if $\rho_y=\sigma_y$. Proofs of this fact can be found in \cite[Prop. 2.6, Prop. 4.5]{FRL}; see also \cite[Theorem 1]{F-unlikely-int}, which among other things proves that the closely related Arakelov-Zhang pairing is the square of a metric on the space of adelic measures.

It is important that the reader be aware that, whenever $y$ is a non-archimedean place, the kernel $\abs{z-w}_y$ in \eqref{eqn:local-pairing} and throughout this paper must be read as the natural extension to the Berkovich line $\sA^1_y$ of the usual distance, which is denoted by $\sup\{z,w\}$ in the article of Favre and Rivera-Letelier \cite[\S 3.3]{FRL} and as the \emph{Hsia kernel} $\delta(z,w)_\infty$ in the book of Baker and Rumely \cite[\S 4]{BakerRumelyBook}. As the appropriate extension is canonically defined, there is no danger in keeping the `classical' notation, so long as the reader is well aware that this is no longer actually a distance function on $\sP^1_y$; as an important example, if $\zeta$ is the Gauss point of $\sP^1_y$, then $\delta(\zeta,\zeta)_\infty = \sup\{\zeta,\zeta\}=1>0$. (In fact, all non-classical points $\zeta$ of the Berkovich line over $\bC_y$ have positive Hsia kernel $\delta(\zeta,\zeta)_\infty>0$.)

For a generalized adelic measure $\rho$ and a discrete probability measure $\Delta$ on $\bP^1(\Qbar)$, we define the \emph{height} of $\Delta$ associated to $\rho$, when it exists, to be:
\begin{equation}\label{eqn:height-defn}
 h_\rho(\Delta) = \frac12 \int_Y (\rho_y - \Delta, \rho_y - \Delta)_y\,d\mu(y).
\end{equation}
When $\rho_y=\lambda_y$ is the standard measure for every $y\in Y$, then we will show that this definition extends the notion of the Weil height to discrete probability measures on $\bP^1(\Qbar)$. In fact, this extends linearly, as we will show in Proposition \ref{prop:height-average} below; however, this fact is not immediate, as we must prove the vanishing of discriminant-type terms over discrete probability measures which may have infinite support, and existence and finiteness of these quantities is nontrivial.

Notice that $\Delta$ defines a well-defined discrete probability measure in $\sP^1_y$ for every place $y\in Y$, as each $y$ defines an embedding of $\Qbar$ into $\bC_y=\bC_p$ where $y\in Y(\bQ,p)$. We note one key difference between our definition and that of Favre and Rivera-Letelier: Favre and Rivera-Letelier define the height as a sum over the places of $K$ where the adelic measure $\rho$ is defined over $K$, and for any $\al\in \Qbar$ for which one wants to compute the height, one first forms the measure:
\[
 [\al]_K = \frac{1}{\# (G_K\cdot \al)} \sum_{z\in G_K \cdot \al} \delta_z,
\]
where $G_K= \Gal(\overline{K}/K)$ denotes the absolute Galois group over $K$ and $\delta_z$ denotes the point mass at $z$. In this fashion, the measure $[\al]_K$ is a probability measure on $\Qbar$ which is in fact invariant under $G_K$, and therefore defines for any single place $v\in M_K$ the same measure for every $y\in Y(K,v)$. They then define $h_\rho(\al)$ to be given by what we denote as $h_\rho([\al]_K)$. This has the convenience of not requiring them to change the places over which they are summing by using $K(\al)$ instead of $K$, however, as we will be dealing with generalized adelic measures, which are limits of certain adelic measures that can be defined over increasing towers of number fields, it is neither practical nor helpful to use this trick. So we will define instead, for $\rho$ a generalized adelic measure and for any $\al\in \Qbar$,
\begin{equation}
 h_\rho(\al) = h_\rho(\delta_\al),
\end{equation}
where $\delta_\al$ is the point mass at $\al$. As we compute the height via \eqref{eqn:height-defn}, we note that the embedding of $\delta_\al$ as a measure in $\sA^1_y$ will change depending on the place $y$ of $\Qbar$.

We now want to show that $h_\rho$ for generalized adelic measures $\rho$ shares largely the same properties as the heights $h_\rho$ for adelic measures $\rho$. We start with a definition which will characterize the discrete measures which have well-defined heights.

\begin{defn}\label{defn:Delta-height-bounded}
 Let $\Delta$ be a discrete probability measure on $\bP^1(\Qbar)$. We say that $\Delta$ \emph{has finite height} if 
 \[
  \int_Y -(\lambda_y, \Delta)_y \,d\mu(y) =  \int_Y \int_{\sA^1_y} \log^+ \abs{z}_y \,d\Delta(z)\,d\mu(y) <\infty.
 \]
 Notice that the integrand is nonnegative, and further, we have explicitly excluded any support at $\infty\in \bP^1(\Qbar)$ from the integral.
\end{defn}
\begin{rmk}
We remind the reader that a discrete measure is one which can be written as a weighted sum of at most countably many Dirac measures, however, in any given completion, the support set of the measure need not be a discrete set in the usual topological sense.
\end{rmk}

As the integrand $\log^+\,\abs{z}_y$ in Definition \ref{defn:Delta-height-bounded} is a compactly supported locally constant function when $\Delta=\delta_\al$ (locally constant on the sets $Y(\bQ(\al),v)$, for $v$ a place of $\bQ(\al)$), this is also true more generally when $\Delta$ is finite supported. In these cases the integral becomes a finite sum. It follows that we have the following lemma:
\begin{lemma}\label{lemma:finite-Delta-have-finite-height}
 If $\Delta$ is a finitely supported probability measure on $\bP^1(\Qbar)$, then $\Delta$ has finite height.
\end{lemma}

\subsection{Existence results for finite height measures}
Our goal is now to justify the definition of `finite height' by proving:
\begin{prop}\label{prop:rho-height-exists-well-defined}
 The height $h_\rho(\Delta)$ defined in \eqref{eqn:height-defn} exists and is finite for every generalized adelic measure $\rho$ and every discrete probability measure $\Delta$ on $\bP^1(\Qbar)$ which has finite height. 
\end{prop}
Before proving Proposition \ref{prop:rho-height-exists-well-defined}, we need some auxiliary results, which we will use later.
\begin{lemma}\label{lemma:p-is-g-plus-log}
 Suppose $y$ is a place of $\Qbar$ and $g : \sP^1_y \ra \bR$ is a continuous function such that $g(\infty)=0$ and $\Delta g = \rho_y - \lambda_y$. Then the potential function 
 \begin{equation}\label{eqn:l1-3}
  p_{\rho_y}(z) = \int_{\sP^1_y} \log\,\abs{z-w}_y \,d\rho_y(w) 
 \end{equation}
 exists and is given by
 \[
  p_{\rho_y}(z) = g(z) + \log^+\,\abs{z}_y.
 \]
\end{lemma}
\begin{proof}
 Existence of the potential $p_{\rho_y}$ follows from Lemmas 2.4 and 4.3 of \cite{FRL} in the archimedean and non-archimedean cases, respectively. To see the equality, observe that 
 \[
 \Delta p_{\rho_y}(z) = \rho_y - \delta_\infty
 \quad\text{and}\quad \Delta \log^+\,\abs{z}_y = \lambda_y - \delta_\infty,  
 \]
 where $\delta_\infty$ denotes the point mass at $\infty$. Thus $\Delta(p_{\rho_y}(z) - g(z) - \log^+\,\abs{z}_y) = 0$, and so the function $p_{\rho_y}(z) - g(z)-\log^+\,\abs{z}_y$ is harmonic on all of $\sP^1_y$, but this means it is constant. By our choice of normalization $g(\infty)=0$, it follows by considering the asymptotics as $z\ra\infty$ that $p_{\rho_y}(z) = g(z) + \log^+ \abs{z}_y$, as claimed.
\end{proof}

\begin{lemma}\label{lemma:lower-bds-on-standard-pairings}
 Let $\Delta$ be a discrete probability measure on $\Qbar$. Then for every finite place $y$ of $\Qbar$, we have
 \[
  (\lambda_y - \Delta, \lambda_y - \Delta)_y \geq 0.
 \]
  For every infinite place $y$ of $\Qbar$, we have
 \[
  (\lambda_y - \Delta, \lambda_y - \Delta)_y \geq -\log 2.
 \]
\end{lemma}
\begin{proof}
 Write
 \[
  \Delta = \sum_{i=1}^\infty t_i \delta_{\alpha_i},
 \]
 where $\alpha_i$ are distinct points of $\Qbar$, which, under the embedding given by our choice of place $y$ we now think of as points of $\bP^1(\bC_y)$. Since we assume $\Delta$ is a probability measure, we have $t_i\geq 0$ for each $i$, and $\sum_i t_i = 1$. By the linearity of the pairing,
 \[
  (\lambda_y - \Delta, \lambda_y - \Delta)_y = (\lambda_y,\lambda_y)_y - 2(\lambda_y, \Delta) + (\Delta,\Delta)_y
 \]
 provided the terms on the right-hand side exist. We will show that the terms on the right side must be a sum of nonnegative terms, which is enough to prove our result. First note that $(\lambda_y,\lambda_y)_y = 0$, so we can ignore that term. The second term is 
 \[
  - 2(\lambda_y, \Delta) = 2\sum_{i=1}^\infty t_i \log^+\,\abs{\alpha_i}_y,
 \]
 and 
 \[
  (\Delta,\Delta)_y = -\sum_{i\neq j} t_i t_j \log\,\abs{\alpha_i - \alpha_j}_y. 
 \]
 Using $\sum_i t_i= 1$, we write
 \[
  2\sum_{i=1}^\infty t_i \log^+\,\abs{\alpha_i}_y = 2\sum_{i=1}^\infty t_i \bigg(\sum_{j=1}^\infty t_j\bigg) \log^+\,\abs{\alpha_i}_y = 2\sum_{i=1}^\infty\sum_{j=1}^\infty t_i t_j \log^+\,\abs{\alpha_i}_y.
 \]
 Now, for any $z,w\in \Qbar$, 
 \[
  \abs{z-w}_y \leq \begin{cases}
                    2 \max\{\abs{z}_y,\abs{w}_y\} &\text{if }y\mid \infty,\\
                    \max\{\abs{z}_y,\abs{w}_y\} &\text{if }y\nmid \infty,
                   \end{cases}
 \]
 so 
 \begin{equation*}
  -\log\, \abs{z-w}_y \geq \begin{cases}
                    -\log 2 - \log^+\,\abs{z}_y - \log^+\, \abs{w}_y &\text{if }y\mid \infty,\\
                    -\log^+\, \abs{z}_y - \log^+\, \abs{w}_y &\text{if }y\nmid \infty,
                   \end{cases}
 \end{equation*}
 or more simply, if we let $\chi_\infty(y)$ be the characteristic function of $Y(\bQ,\infty)$, 
 \begin{equation}\label{eqn:lower-bound-for-chordal}
  \log^+\,\abs{z}_y + \log^+\, \abs{w}_y -\log\, \abs{z-w}_y \geq -\chi_\infty(y)\log 2.
 \end{equation}
 It follows that 
\begin{align*}
  - 2(\lambda_y, \Delta) + (\Delta,\Delta)_y
    &= \sum_{i\neq j} t_it_j(\log^+\,\abs{\alpha_i}_y + \log^+\,\abs{\alpha_j}_y - \log\,\abs{\alpha_i-\alpha_j}_y)\\
    &\hspace{10mm} + \sum_{i=1}^\infty 2t_i^2 \log^+\,\abs{\alpha_i}_y\\
    &\geq - \sum_{i\neq j} t_it_j\chi_\infty(y)\log 2\\
    &\geq -\chi_\infty(y)\log 2,
 \end{align*}
 and the conclusion follows.
\end{proof}
\begin{rmk}
 We note that it is possible to improve the lower bound in the preceding lemma in the archimedean case; in particular, in the case where $\Delta$ is equally supported on a finite set, we could apply Mahler's inequality \cite{Mahler} and Baker's sweeping generalization of Mahler's result to lower bounds for averages of Green's functions \cite{BakerAverages}. However, these results are proven in a more limited context, and as we will not need a stronger result, we will not make any effort to improve the result here with those ideas.
\end{rmk}

We now prove a generalization of the product formula for the integral of discriminant-type terms over all places of $\Qbar$:
\begin{lemma}\label{lemma:discrs-int-to-0}
 Suppose that $\Gamma,\Delta$ are discrete probability measures on $\bP^1(\Qbar)$ which have finite height. Then 
 \[
  \int_Y (\Gamma,\Delta)_y \,d\mu(y) = 0.
 \]
\end{lemma}
\begin{proof}
 The proof is essentially the product formula applied term by term; the main difficulty lies in justifying the interchange of the sum and the integral in order to apply the product formula to each individual difference of terms which appears. To do this, we will show that there exists a function in $L^1(Y)$ which we can add to the function $(\Gamma,\Delta)_y$ to make it nonnegative as a function of the place $y$. 
 
 We start by writing, as $\Gamma$ and $\Delta$ are discrete probability measures,
 \[
  \Gamma = \sum_{n=1}^\infty s_n \delta_{z_n}\quad\text{and}\quad \Delta = \sum_{n=1}^\infty t_n \delta_{w_n},
 \]
 where $z_n,w_n\in \Qbar$, $\delta_{z}$ denotes the Dirac point mass at $z\in\Qbar$, and $\sum_n s_n = \sum_n t_n = 1$. By the definition of the energy pairing,
 \[
  (\Gamma,\Delta)_y = \int_{\sA^1_y \times \sA^1_y\setminus \mathrm{Diag}_y} -\log\,\abs{z-w}_y\,d\Gamma(z)\,d\Delta(w) 
  = \sum_{\substack{m,n\\ z_m\neq w_n}} -s_m t_n\log\,\abs{z_m-w_n}_y,
 \]
 where $\mathrm{Diag}_y = \{ (z,z)\in \sA^1_y \times \sA^1_y : z\in \bC_y\}$ denotes the diagonal of classical points. (This makes no difference for an archimedean place $y$, but for the Berkovich line $\sA^1_y$ over a nonarchimedean place, we are only excluding the diagonal of Type I points.) By the assumption that $\Gamma$ and $\Delta$ have finite height, we have that
 \[
  \int_Y \sum_{n} s_n \log^+\,\abs{z_n}_y \,d\mu(y)<\infty\quad\text{and}\quad 
  \int_Y \sum_{n} t_n \log^+\,\abs{w_n}_y \,d\mu(y)<\infty. 
 \]
 Let $\chi_\infty$ be the characteristic function of $Y(\bQ,\infty)$. We define the functions
 \[
  f(y) = \sum_{n} s_n \log^+\,\abs{z_n}_y 
  \quad\text{and}\quad 
  g(y) = \sum_{n} t_n \log^+\,\abs{w_n}_y.
 \]
 Note that $f,g\geq 0$ and $f,g\in L^1(Y)$ by our assumption that $\Gamma$ and $\Delta$ have finite height. We also have $\chi_\infty \in L^1(Y)$, as $\chi_\infty$ is a constant function on its support, which has finite measure $\mu(Y(\bQ,\infty)) = 1$. Notice that 
 \begin{align*}
  \chi_\infty(y) \log 2 + f(y) + g(y) + (\Gamma,\Delta)_y \\
 &\hspace{-45mm}= \chi_\infty(y) \log 2 + \sum_{m} s_m \log^+\,\abs{z_m}_y \\
 &\hspace{-40mm}+ \sum_{n} t_n \log^+\,\abs{w_n}_y -\sum_{\substack{m,n\\ z_m\neq w_n}} s_m t_n\log\,\abs{z_m-w_n}_y\\
 &\hspace{-45mm}= \sum_{m,n} s_m t_n \chi_\infty(y) \log 2 + \sum_{m,n} s_m t_n \log^+\,\abs{z_m}_y \\
 &\hspace{-40mm}+ \sum_{m,n}s_m t_n \log^+\,\abs{w_n}_y - \sum_{\substack{m,n\\ z_m\neq w_n}} s_m t_n\log\,\abs{z_m-w_n}_y,
 \end{align*}
 where we have used the fact that $\sum_m s_m = \sum_n t_n = 1$ above to insert the extra copies of $s_m$ and $t_n$. Thus, applying Lemma \ref{lemma:lower-bds-on-standard-pairings}, we have
  \begin{align*}
  \chi_\infty(y) \log 2 + f(y) + g(y) + (\Gamma,\Delta)_y 
  \\
 &\mbox{}\hspace{-45mm}= \sum_{\substack{m,n\\ z_m = w_n}} s_m t_n \big( 2\log^+\,\abs{z_m}_y + \chi_\infty(y) \log 2\big)\\
  &\mbox{}\hspace{-40mm}+ \sum_{\substack{m,n\\ z_m\neq w_n}} s_m t_n\big( \chi_\infty(y) \log 2 + \log^+\,\abs{z_m}_y + \log^+\,\abs{w_n}_y - \log\,\abs{z_m-w_n}_y\big)\\
  &\mbox{}\hspace{-45mm}\geq 0,
 \end{align*}
 where nonnegativity is clear for the first term and follows from \eqref{eqn:lower-bound-for-chordal} for the second term. As we have added $L^1$-functions and the resulting function is nonnegative, and since our measure spaces are $\sigma$-finite, we can apply Tonelli's theorem to justify the interchange of summation and integration in the original integral we wished to compute: 
 \begin{align}
 \begin{split}
  \int_Y (\Gamma,\Delta)_y \,d\mu(y) &= \int_Y -\sum_{\substack{m,n\\ z_m\neq w_n}} s_m t_n \log\,\abs{z_m-w_n}_y \,d\mu(y)\\
  &= - \sum_{\substack{m,n\\ z_m\neq w_n}}s_m t_n \int_Y  \log\,\abs{z_m-w_n}_y \,d\mu(y).
 \end{split}
 \end{align}
Notice however that for every $\alpha \in \Qbar$, if we let $K = \bQ(\al)$, then 
 \[
  \int_Y \log\, \abs{\al}_y \,d\mu(y) = \sum_{v\in M_K} \mu(Y(K,v)) \log\, \abs{\al}_v = \sum_{v\in M_K} \frac{[K_v:\bQ_v]}{[K:\bQ]} \log\, \abs{\al}_v = 0
 \]
by the product formula. It follows that 
\[
 \int_Y  \log\,\abs{z_m-w_n}_y \,d\mu(y) = 0
\]
 for every $m,n$ with $z_m \neq w_n$, and our proof is complete. 
\end{proof}

Using Lemma \ref{lemma:discrs-int-to-0} we can now prove a result that justifies the use of the term `finite height' in Definition \ref{defn:Delta-height-bounded}:
\begin{prop}\label{prop:Delta-finite-Weil-height}
 Let $\Delta$ be a discrete probability measure on $\Qbar$. Then $\Delta$ has finite height in the sense of Definition \ref{defn:Delta-height-bounded} if and only if $h(\Delta)$ exists and is finite.
\end{prop}
\begin{proof}
 First, assume that $\Delta$ has finite height. Then by the bilinearity of the energy pairing, 
 \[
  h(\Delta) = \frac12 \int_Y (\Delta - \lambda_y, \Delta - \lambda_y)_y \,d\mu(y) = \frac12 \int_Y (\lambda_y, \lambda_y)_y - 2(\Delta,\lambda_y)_y + (\Delta,\Delta)_y \,d\mu(y),
 \]
 presuming these quantities exist. But $(\lambda_y, \lambda_y)_y= 0$ for all $y$, 
 \[
 - 2\int_Y (\Delta,\lambda_y)_y \,d\mu(y) = 2 \int_Y \int_{\sA^1_y} \log^+\,\abs{z}_y \,d\Delta(z)d\mu(y) < \infty
 \]
 by the assumption that $\Delta$ has finite height, and by Lemma \ref{lemma:discrs-int-to-0},
 \[
  \int_Y (\Delta,\Delta)_y\,d\mu(y) = 0.
 \]
  
 For the other direction, we now suppose that $h(\Delta)<\infty$. If $\Delta$ is finitely supported, the result is trivial, so assume that $\Delta$ has countably infinite support. Write 
 \[
  \Delta = \sum_{i=1}^\infty t_i \delta_{\alpha_i}. 
 \]
 Then by Lemma \ref{lemma:lower-bds-on-standard-pairings}, we have
 \[
  (\Delta - \lambda_y, \Delta - \lambda_y)_y \geq -\chi_\infty(y)\log 2,
 \]
 where $\chi_\infty$ is the characteristic function of $Y(\bQ,\infty)$.
 Since $\chi_\infty\in L^1(Y)$, Tonelli's theorem applies to the integral defining the height $h(\Delta)$, and we can exchange the order of integration, and in particular, as we noted in the proof of Lemma \ref{lemma:discrs-int-to-0}, the product formula applied to the discriminant-type term yields
 \[
  \int_Y \log\,\abs{\alpha_i-\alpha_j}_y\,d\mu(y) = 0.
 \]
 It follows, since $h(\Delta)<\infty$, that
 \begin{equation}\label{eqn:finite-std-height-is-log-plus-terms}
  h(\Delta) = \frac12 \int_Y (\Delta - \lambda_y, \Delta - \lambda_y)_y\,d\mu(y) = \int_Y\int_{\sA^1_y}\log^+\,\abs{z}_y\,d\Delta(z)\,d\mu(y)< \infty,
 \end{equation}
 and so $\Delta$ meets the finite height criterion.
\end{proof}
It will follow from Proposition \ref{prop:Delta-finite-Weil-height} and Theorem \ref{thm:gen-h-is-weil-nonneg} below, which shows that for every generalized adelic measure $\rho$, $h_\rho = h + O(1)$ with big-$O$ constant depending only on $\rho$, that the result of Proposition \ref{prop:Delta-finite-Weil-height} is in fact true for all heights $h_\rho$ associated to a generalized adelic measure.

\begin{lemma}\label{lemma:pairing-terms-bdd}
 Suppose $y$ is a place of $\Qbar$ and $\rho_y$ is a Borel probability measure on $\sP^1_y$ for which there exists a continuous potential $g_y : \sP^1_y \ra \bR$ such that $g_y(\infty) = 0$ and that, with respect to the standard measure $\lambda_y$, we have $\Delta g_y = \rho_y - \lambda_y$. Then
 \[
 -2 \sup_{z\in \sP^1_y} \abs{g_y(z)} \leq (\rho_y,\rho_y)_y \leq 4 \sup_{z\in \sP^1_y} \abs{g_y(z)}.
 \]
 In particular, for an adelic measure $\rho$, $\abs{(\rho_y,\rho_y)_y} \leq 4 C(y)$ for $\mu$-almost every $y\in Y$.
\end{lemma}
\begin{proof}
 Let
 \[
  C(y) = \sup_{z\in \sP^1_y} \abs{g_y(z)},
 \]
 as in the definition of an adelic measure. If follows from \cite[Prop. 2.6]{FRL} when $y$ is archimedean and \cite[Prop 4.5]{FRL} when $y$ is non-archimedean that
 \[
  0 \leq (\rho_y - \lambda_y, \rho_y - \lambda_y)_y = \int_{\sP^1_y} -g_y(z) \, d(\rho_y -\lambda_y)(z). 
 \]
 It follows that
 \begin{equation}\label{eqn:lemma1}
  0 \leq (\rho_y - \lambda_y, \rho_y - \lambda_y)_y \leq \abs{\rho_y - \lambda_y}(\sP^1_y) \cdot \sup_{z\in \sP^1_y} \abs{g_y(z)} \leq 2 C(y).
 \end{equation}
 By the linearity and symmetry of the energy pairing, 
 \[
  (\rho_y - \lambda_y, \rho_y - \lambda_y)_y = (\rho_y,\rho_y)_y - 2 (\rho_y,\lambda_y)_y + (\lambda_y, \lambda_y)_y.
 \]
 As before, we have $(\lambda_y, \lambda_y)_y= 0$. If we can bound the mixed term $(\rho_y,\lambda_y)_y$, our result will follow. By Lemma \ref{lemma:p-is-g-plus-log}, the potential function 
 \[
  p_{\rho_y}(z) = \int_{\sP^1_y} \log\,\abs{z-w}\,d\rho_y(w)
 \]
 of $\rho_y$ is $p_{\rho_y}(z) = g_y(z) + \log^+\,\abs{z}_y$. 
 We now apply Lemmas 2.4 and 2.5 of \cite{FRL} to say that 
 \begin{equation}\label{eqn:lemma2}
  \abs{(\rho_y , \lambda_y)} \leq \int_{\sP^1_y} \abs{p_{\rho_y}(z)} \,d\lambda_y(z) \leq \int_{\sP^1_y} (\log^+\,\abs{z}_y + C(y)) \,d\lambda_y(z) = C(y).
 \end{equation}
 Combining equations \eqref{eqn:lemma1} and \eqref{eqn:lemma2} the bounds follow.
\end{proof}

We are now ready to prove Proposition \ref{prop:rho-height-exists-well-defined}, which states the $h_\rho(\Delta)$ exists and is finite for every generalized adelic measure $\rho$ and every discrete probability measure $\Delta$ on $\bP^1(\Qbar)$ which has finite height in the sense of Definition \ref{defn:Delta-height-bounded}.
\begin{proof}[Proof of Proposition \ref{prop:rho-height-exists-well-defined}]
 By the linearity of the energy pairing, write:
\begin{equation}\label{eqn:h-breakdown-into-3-terms}
 \begin{split}
 h_\rho(\Delta) &= \frac12 \int_Y (\rho_y - \Delta, \rho_y - \Delta)_y\,d\mu(y)\\
  &= \frac12\int_Y (\rho_y,\rho_y) - 2(\Delta, \rho_y)  +(\Delta, \Delta)_y\,d\mu(y).
 \end{split}
 \end{equation}
 We will prove first that each of the three integrands in \eqref{eqn:h-breakdown-into-3-terms} above is an integrable function of the place $y$, and then that each of the three terms above is bounded above and below by an integrable function on $Y$, so that the integral exists, as claimed. 
  
We start with the first term. Measurability of $(\rho_y,\rho_y)$ as a function of $y$ follows from the measurability of the function $g : Y(\bQ,p)\times \sP^1(\bC_p)\ra\bR$ for every rational prime $p$. To see that the first term has finite integral, we use Lemma \ref{lemma:pairing-terms-bdd} to say that 
 \[
  \bigg| \frac12 \int_Y (\rho_y,\rho_y)_y \,d\mu(y)\bigg| \leq \frac12 \int_Y 4 C(y) \,d\mu(y) <\infty
 \]
 by our assumption that $C(y)\in L^1(Y)$. 

Consider now the second term in the integral,
\[
 \int_Y -(\rho_y,\Delta)_y \,d\mu(y) = \int_Y \int_{\sA^1_y} p_{\rho_y}(z) \,d\Delta(z)\,d\mu(y),
\]
where $p_{\rho_y}$ denotes the potential function of $\rho_y$ given by
\[
 p_{\rho_y}(z) = \int_{\sA^1_y} \log\abs{z-w}_y \,d\rho_y(w).
\]
Since each point in the support of $\Delta$ is a single algebraic number, its $y$-adic absolute value defines a locally constant function of $y$, so measurability is trivial. Here we make use of the assumption that $\Delta$ has finite height, which gives
\[
 0 \leq \int_Y \int_{\sA^1_y} \log^+\, \abs{z}_y \,d\Delta(z)\,d\mu(y)<\infty, 
\]
and now we need to bound the integral with $p_{\rho_y}(z)$ in place of $\log^+\,\abs{z}_y$. By Lemma \ref{lemma:p-is-g-plus-log}, we know that 
\[
 p_{\rho_y}(z) = g_y(z) + \log^+\,\abs{z}_y,
\]
so the difference between the two integrals is bounded by 
\begin{equation}\label{eqn:bound-bw-p-and-log-plus-on-Delta}
 \int_Y \int_{\sA^1_y} \abs{g_y(z)} \,d\Delta(z)\,d\mu(y) \leq \int_Y \sup_{z\in\sP^1_y} \abs{g_y(z)}\,d\mu(y) = \int_Y C(y)\, d\mu(y)<\infty
\end{equation}
 as $C(y)\in L^1(Y)$.
 
 The third term in the integrand, namely
\[
 \frac12\int_Y (\Delta,\Delta)_y\,d\mu(y),
\]
 is equal to $0$ by Lemma \ref{lemma:discrs-int-to-0}, since $\Delta$ is a discrete probability measure of finite height. We have shown that all three terms are integrable, so we are done.
\end{proof}

We note that, in the case that $\rho$ is defined over a single number field $K/\bQ$, then we have the usual Galois invariance of the height:
\begin{prop}\label{prop:galois-invariance-of-height}
 Suppose that $\rho$ is defined over a single number field $K/\bQ$. Then the height $h_\rho$ is $\Gal(\overline{K}/K)$-invariant in the sense that if we define $\sigma\Delta$ by linearly extending the action $\sigma(\delta_\al) = \delta_{\sigma\al}$ for all $\al\in\Qbar$, then 
 \begin{equation}
  h_\rho(\sigma\Delta) = h_\rho(\Delta) \quad\text{for all}\quad\sigma\in \Gal(\overline{K}/K).
 \end{equation}
 In particular, if $[\al]_K$ denotes the probability measure on $\Qbar$ equally distributed on each $K$-Galois conjugate of $\al$, then
 \begin{equation}
  h_\rho(\al) = h_\rho([\al]_K).
 \end{equation}
\end{prop}
\noindent We remind the reader that the action of the absolute Galois group on the space of places $Y$ is defined by $\abs{\alpha}_{\sigma y} = \abs{\sigma^{-1}(\alpha)}_y$, so that the action on $Y$ is compatible with the action on $\Qbar$; we use the inverse to ensure that
\[
    \abs{\alpha}_{\tau\sigma y} =     \abs{\tau^{-1}(\alpha)}_{\sigma y} = \abs{\sigma^{-1}\tau^{-1}(\alpha)}_y = \abs{(\tau\sigma)^{-1}(\alpha)}_y.
\]

\begin{proof}
By the linearity of the energy pairing, we have
\begin{equation*}
 \begin{split}
 h_\rho(\sigma\Delta) &= \frac12 \int_Y (\rho_y - \sigma\Delta, \rho_y - \sigma\Delta)_y\,d\mu(y)\\
  &= \frac12\int_Y (\rho_y,\rho_y) - 2(\sigma\Delta, \rho_y)  +(\sigma\Delta, \sigma\Delta)_y\,d\mu(y).
 \end{split}
 \end{equation*}
By Lemma \ref{lemma:discrs-int-to-0},
$\int_Y (\sigma\Delta, \sigma\Delta)_y d\mu(y) = 0$ so it suffices to show that 
\begin{equation*}
    \int_Y (\sigma\Delta, \rho_y)_y d\mu(y) = \int_Y (\Delta, \rho_y)_y d\mu(y).
\end{equation*}
Writing $\Delta = \sum\limits_{i} t_i\delta_{\alpha_i}$ where $\sum\limits_i t_i = 1$, and using the fact that $\rho$ is defined over $K$, we have 
\begin{align*}
\int_Y (\sigma\Delta, \rho_y)_y d\mu(y)&= \int_{Y_K} \int_{\sA^1\times \sA^1 \setminus \Diag_y} -\log |z-w|_y d \sigma\Delta (z) d\rho_y (w)d\mu (y)\\
&= \int_{Y_K} \int_{\sA^1_y} -p_{\rho_y}(z) \sigma\Delta (z)d\mu (y)\\
&= \int_{Y_K}\sum\limits_{i} -t_i\,  p_{\rho_y}(\sigma \alpha_i)d\mu (y).
\end{align*}
Since $p_{\rho_y}(\sigma \alpha_i) = p_{\rho_{\sigma^{-1}y}}(\alpha_i)$, then we have that 
\begin{equation*}
 \begin{split}
    \int_Y (\sigma\Delta, \rho_y)_y d\mu(y) &= \int_{Y_K} \sum\limits_{i} t_i p_{\rho_y}(\alpha_i)d\mu (y)\\
    &= \int_Y (\Delta, \rho_y)_y d\mu(y),
 \end{split}
 \end{equation*}
as desired. 
\end{proof}

We conclude this section with a lemma which will be needed to justify appplications of Tonelli's theorem below:
\begin{lemma}\label{lemma:integrability-of-delta-heights}
 Let $\Delta$ be a discrete probability measure on $\Qbar$ and $\rho$ a generalized adelic measure. Let $C(y)$ be the associated local bound on the Green's function associated to $\rho$ at the place $y$. Then 
 \[
  (\rho_y-\Delta,\rho_y-\Delta)_y \geq - 4C(y) - \chi_\infty(y) \log 2,
 \]
 where $\chi_\infty(y)$ denotes the characteristic function of $Y(\bQ,\infty)$.
\end{lemma}
\begin{proof}
 Our goal will be to give a lower bound for
 \begin{equation}\label{eqn:24-main}
  (\rho_y-\Delta,\rho_y-\Delta)_y - (\lambda_y-\Delta,\lambda_y-\Delta)_y = (\rho_y,\rho_y)_y - 2 (\rho_y-\lambda_y,\Delta)_y
 \end{equation}
 and use Lemma \ref{lemma:lower-bds-on-standard-pairings} to bound the pairing of $\Delta$ with $\lambda_y$. Now, Lemma \ref{lemma:pairing-terms-bdd} gives that 
 \begin{equation}\label{eqn:24-inf-height}
  (\rho_y,\rho_y)_y \geq - 2C(y),
 \end{equation}
and by Lemmas 2.5 and 4.4 of \cite{FRL},
 \begin{equation}
 - 2(\rho_y-\lambda_y,\Delta)_y = 2 \int_{\sA^1_y} g_y(z)\,d\Delta(z),
 \end{equation}
 so
 \begin{equation}\label{eqn:24-bound-on-mixed}
 \abs{- 2 (\rho_y-\lambda_y,\Delta)_y}\leq 2\sup_{z\in\sA^1_y} \abs{g_y(z)} = 2 C(y).
 \end{equation}
 Now, Lemma \ref{lemma:lower-bds-on-standard-pairings} yields
 \[
  (\lambda_y-\Delta,\lambda_y-\Delta)_y \geq -\chi_\infty(y)\log 2,
 \]
 so combining this with the results of \eqref{eqn:24-inf-height} and \eqref{eqn:24-bound-on-mixed} in equation \eqref{eqn:24-main} gives the result.
\end{proof}

\section{Equidistribution for generalized adelic measures}\label{sec:equi-for-gen-adelic}

\subsection{Heights associated to generalized adelic measures are Weil heights}
We now extend \cite[Th\'eor\`eme 1]{FRL} to generalized adelic measures:
\begin{thm}\label{thm:gen-h-is-weil-nonneg}
 Let $\rho$ be a generalized adelic measure. Then:
 \begin{enumerate}
  \item \label{item:h-rho-is-weil}$h_\rho(\Delta) = h(\Delta) + O(1)$ for every discrete probability measure $\Delta$ on $\bP^1(\Qbar)$ with finite height, where the big-$O$ constant is independent of $\Delta$ and depends only on $\rho$, and 
 \item \label{item:h-rho-nonneg}$h_\rho$ is essentially nonnegative, in the sense that for any $\ep>0$, the set \[\{\al\in \bP^1(\Qbar) : h_\rho(\al) < -\ep \}\] is finite.
 \end{enumerate}
\end{thm}
\begin{proof}
 We start with the proof of the first statement. This is quite similar in spirit to the proof of boundedness in Proposition \ref{prop:rho-height-exists-well-defined} above. We observe that:
 \[
  h_\rho(\Delta) = \frac12 \int_Y (\rho_y -\Delta, \rho_y- \Delta)_y \,d\mu(y)
 \]
 and 
 \[
  h(\Delta) = \frac12 \int_Y (\lambda_y - \Delta, \lambda_y - \Delta)_y \,d\mu(y).
 \]
 Further, by Proposition \ref{prop:Delta-finite-Weil-height}, $h(\Delta)$ exists and is finite by our assumption on $\Delta$. Since $(\lambda_y,\lambda_y)_y = 0$ at every place,
 \begin{equation}\label{eqn:height-bound}
  \abs{ h_\rho(\Delta) - h(\Delta)} \leq \bigg| \frac12 \int_Y (\rho_y, \rho_y)_y\,d\mu(y) \bigg| + \bigg| \int_Y (\rho_y-\lambda_y, \Delta)_y\,d\mu(y) \bigg |.
 \end{equation}
 Now, by our assumption that $\rho$ is a generalized adelic measure, we know that there exists a measurable function $g : Y(\bQ,p) \times \sP^1(\bC_p) \ra \bR$ at every rational prime $p$ such that, if $g_y(z) = g(y,z)$, then for $\mu$-almost all $y$, we have that $g_y$ is continuous and $\Delta g_y(z) = \rho_y - \lambda_y$, and further, the function
 \[
  C(y) = \sup_{z\in \sP^1_y} \abs{g_y(z)}
 \]
 satisfies $C\in L^1(Y)$. Now, by Lemmas 2.5 and 4.4 of \cite{FRL}, 
 \[
  (\rho_y-\lambda_y, \Delta)_y = \int_{\sA^1_y} g_y(z)\,d\Delta(z),
 \]
 so it follows that $\abs{(\rho_y-\lambda_y, \Delta)_y}\leq C(y)$, independent of $\Delta$. It follows that $h = h_\rho + O(1)$, and indeed, using \eqref{eqn:height-bound} and applying Lemma \ref{lemma:pairing-terms-bdd} to bound the first term in \eqref{eqn:height-bound}, we can explicitly compute a big-$O$ constant, which is independent of $\Delta$:
 \begin{equation}\label{eqn:explicit-height-bound}
  \abs{ h_\rho(\Delta) - h(\Delta)} \leq 3 \int_Y C(y)\,d\mu(y).
 \end{equation}
 In order to prove our height function is essentially nonnegative, assume for the sake of contradiction that, for some $\ep>0$, the set 
 \[
  \{\al\in\Qbar : h_\rho(\al) < -\ep \}
 \]
 is instead infinite. Let $(\al_n)_{n=1}^\infty$ be sequence of distinct elements in this set and let $\Delta_N$ be the sequence of discrete probability measures on $\Qbar$ given by
 \[
  \Delta_N = \frac1N \sum_{n=1}^N \delta_{\al_n}.
 \]
 As the sums involved in computing the height are finite, it follows from applying the product formula to cancel the discriminant terms that
 \[
  h_\rho(\Delta_N) = \frac1N \sum_{n=1}^N h_\rho(\al_n),
 \]
  and so $h_\rho(\Delta_N) < -\ep$ as well. However, for every place $y$ of $\Qbar$, it follows from applying Propositions 2.8 and 4.9 of \cite{FRL} with $\ep = 1/N$ that 
 \begin{equation}\label{eqn:degree-to-inf-means-nonneg-pairing}
  \liminf_{N\ra\infty} (\rho_y - \Delta_N, \rho_y - \Delta_N)_y \geq 0.
 \end{equation}
 Now notice that by Lemma \ref{lemma:integrability-of-delta-heights}, the functions
 \[
  f_n(y) = (\rho_y - \Delta_N, \rho_y - \Delta_N)_y 
 \]
 are almost nonnegative in the sense that $f_n(y) \geq - 4C(y) + \chi_\infty(y)\log 2$, where $\chi_\infty$ denotes the characteristic function of $Y(\bQ,\infty)$. As $C$ and $\chi_\infty$ are both in $L^1(Y)$, it follows that Fatou's lemma applies to the function $f_n(y)$ and we can say that
 \begin{equation}
 \begin{split}
  0 &\leq \frac12 \int_Y \liminf_{N\ra \infty} (\rho_y - \Delta_N, \rho_y - \Delta_N)_y \,d\mu(y)\\
  &\leq \liminf_{N\ra \infty} \frac12 \int_Y (\rho_y - \Delta_N, \rho_y - \Delta_N)_y \,d\mu(y)\\
  &= \liminf_{N\ra\infty} h_\rho(\Delta_N) \leq -\ep < 0,
 \end{split}
 \end{equation}
which is a contradiction. It follows that our height is essentially nonnegative. 
\end{proof}
We can now state the corollary of Proposition \ref{prop:Delta-finite-Weil-height} above, which we stated would follow from proving Theorem \ref{thm:gen-h-is-weil-nonneg}:
\begin{cor}\label{cor:finite-height-is-same-for-all-heights}
 Let $\Delta$ be a discrete probability measure on $\bP^1(\Qbar)$ and $\rho$ be a generalized adelic measure. Then $\Delta$ has finite height in the sense of Definition \ref{defn:Delta-height-bounded} if and only if $h_\rho(\Delta)$ exists and is finite.
\end{cor}
\begin{rmk}
Note that it Corollary \ref{cor:finite-height-is-same-for-all-heights} means that if $h_\rho(\Delta)$ exists and is finite for one generalized adelic measure, then it exists and is finite for all generalized adelic measures.
\end{rmk}

The reader may wonder what use we have for allowing arbitrary discrete probability measures on $\Qbar$ rather than looking at point masses on algebraic numbers, or finite probability measures equally supported on the Galois conjugates of an algebraic number of a base field -- especially given that we have gone to some trouble in order to define heights associated to this more general class of probability measures. There are several reasons for this, one being that the equidistribution theorem for stochastic dynamical systems, which we will prove below, naturally is phrased in terms of random backwards orbits of a point under the stochastic system, and when we have infinitely many maps in this system, our backwards orbits measures will generally have infinite support. 

However, there are other reasons why we may wish to do this. One can be observed in the proof of the essential nonnegativity of the height associated to a generalized adelic measure, where, rather than taking Galois orbits over a base number field (which may not exist for a generalized adelic measure!) and applying Northcott's theorem, we proved essential nonnegativity by forming discrete measures that averaged over the algebraic numbers of negative height. The fact that this proof avoided invoking Northcott's theorem seems to indicate it has the potential for broader applications.

Yet another reason that one might be interested in the heights of discrete measures on $\bP^1(\Qbar)$, including the Weil height, is that they behave well in forming averages. The following proposition, which demonstrates this fact, will be a key result in proving our results, and is of, we believe, independent interest:
\begin{prop}\label{prop:height-average}
 Let $\rho$ be a generalized adelic measure and suppose that we are given a (possibly finite) sequence of discrete probability measures $\Delta_1,\Delta_2, \ldots$ on $\bP^1(\Qbar)$, each of finite height, and real numbers $t_1,t_2, \ldots\geq 0$ with $\sum_n t_n= 1$. Then 
 \[
  h_\rho\bigg(\sum_n t_n \Delta_n\bigg) = \sum_n t_n \, h_\rho(\Delta_n). 
 \]
\end{prop}
\begin{rmk}
 We note that in Proposition \ref{prop:height-average} above, we did not need to assume that our measures $\Delta_n$ above have finite height. However, if $h_\rho(\Delta_n)=\infty$ for some $n$ and $t_n>0$, it will follow that $h_\rho\big(\sum_n t_n \Delta_n\big) = \infty$ as well, as we will demonstrate that the integrands are all bounded below for each place $y$ by an $L^1(Y)$ function.
\end{rmk}
\begin{proof}[Proof of Proposition \ref{prop:height-average}]
 We recall that by \eqref{eqn:height-defn},
\begin{align*}
 h_\rho(\Delta_n) &= \frac12 \int_Y (\rho_y - \Delta_n, \rho_y - \Delta_n)_y\,d\mu(y)\\
 &= \int_Y \frac12 (\rho_y, \rho_y)_y  - (\Delta_n, \rho_y)_y + \frac12 (\Delta_n,\Delta_n)_y\,d\mu(y)
\end{align*}
for each $n\geq 1$. The key observation is that, by Lemma \ref{lemma:integrability-of-delta-heights} above, our integrand is within an $L^1$ function of being nonnegative, so we will be able to apply Tonelli's theorem to bring the integrals over all places into the sums involved. This will eliminate the discriminant-type terms -- the only terms in the height which do not \emph{prima facie} appear to linearly average -- and allow us to derive the result. This is less trivial than it may seem at first glance, as we are \emph{not} assuming that all of the heights are finite.

We start by showing that, regardless of whether the height is finite or infinite, we can disregard the discriminant term in its computation. Let $\Delta$ be a given discrete probability measure on $\bP^1(\Qbar)$. As the result is trivial if $\Delta$ is supported on a finite set, we assume that it does not. Write $\Delta$ as 
 \[
  \Delta = \sum_{i=1}^\infty t_i \delta_{\al_i},
 \]
 where $\al_i\in\bP^1(\Qbar)$ are distinct, and let
\[
 \frac1m = \int_{\sA^1_y\times \sA^1_y\setminus\mathrm{Diag}_y } 1\,d\Delta(z)d\Delta(w),
\]
 where again $\mathrm{Diag}_y = \{(z,z)\in \sA^1_y\times \sA^1_y : z\in \bC_y\}$ is the diagonal of classical points. Notice that, since $\Delta$ is not a point mass as we assumed it is not supported on a finite set, $0<m<1$. Let $G(z,w)$ be the function given by
 \[
 G(z,w) = m\big(p_{\rho_y}(z) + p_{\rho_y}(w)\big) - \log\,\abs{z-w}_y + m(\rho_y,\rho_y)_y.
 \]
 Notice that $G$ is a slightly different normalization of the usual Arakelov-Green's function for $\rho_y$, where $p_{\rho_y}(z)$ is the associated potential function. Note that as $\Delta$ is supported on $\bP^1(\Qbar)$, instead of our usual local integral over $\sA^1_y\times \sA^1_y\setminus \mathrm{Diag}_y$, we will integrate over $\bA^1(\Qbar)\times\bA^1(\Qbar)\setminus D$ where $D = \{(\al,\al) : \al\in \Qbar\}$. with this normalization,
 \[
  \int_{\bA^1(\Qbar)\times\bA^1(\Qbar)\setminus D} G(z,w)\,d\Delta(z)d\Delta(w) = (\rho_y-\Delta,\rho_y-\Delta)_y,
 \]
 so 
 \[
  h_\rho(\Delta) = \frac12 \int_Y \int_{\sA^1_y\times \sA^1_y\setminus\mathrm{Diag}_y } G(z,w)\,d\Delta(z)d\Delta(w)\,d\mu(y).
 \]
 Now, by Lemma \ref{lemma:integrability-of-delta-heights},

\begin{align*}
  &\hspace{-5mm}(\rho_y-\Delta,\rho_y-\Delta)_y + 4C(y) + \chi_\infty(y)\log 2\\
  &= \int_{\bA^1(\Qbar)\times\bA^1(\Qbar)\setminus D} G(z,w)\,d\Delta(z)d\Delta(w) + 4C(y) + \chi_\infty(y)\log 2\\
  &= \int_{\bA^1(\Qbar)\times\bA^1(\Qbar)\setminus D} G(z,w) + m(4C(y) + \chi_\infty(y)\log 2)\,d\Delta(z)d\Delta(w)\\
  &\geq 0,  
 \end{align*}
 and since $4C(y) + \chi_\infty(y)\log 2$ is a nonnegative integrand with 
 \[
  \int_Y \int_{\bA^1(\Qbar)\times\bA^1(\Qbar)\setminus D} (4C(y) + \chi_\infty(y)\log 2) \,d\Delta(z)d\Delta(w)\,d\mu(y) < \infty,
 \]
 we can apply Tonelli's theorem to the integrals defining $h_\rho(\Delta)$ and interchange the order of integration:
 \[
  h_\rho(\Delta) = \frac12 \int_{\bA^1(\Qbar)\times\bA^1(\Qbar)\setminus D } \int_Y G(z,w)\,d\mu(y)\,d\Delta(z)d\Delta(w).
 \]
 However, for any fixed $z\neq w$ in the support of $\Delta$, we clearly have
 \[
  \int_Y G(z,w)\,d\mu(y) = \int_Y m(p_{\rho_y}(z) + p_{\rho_y}(w)) - \log\,\abs{z-w}_y + m(\rho_y,\rho_y)_y\,d\mu(y).
 \]
 We write the integrals
 \[
  \int_Y p_{\rho_y}(z)\,d\mu(y) = -\int_Y (\rho_y, \delta_z)_y\,d\mu(y) ,\quad 
  \int_Y p_{\rho_y}(w)\,d\mu(y) = -\int_Y (\rho_y, \delta_w)_y\,d\mu(y) ,
 \]
 and note that the integral of the third term is
 \[
  \int_Y -\log\,\abs{z-w}_y \,d\mu(y) = 0
 \]
 by the product formula, as $z-w\in\Qbar$ is a single algebraic number by our assumptions on $\Delta$. It follows that 
 \begin{equation}
  \begin{split}
   h_\rho(\Delta) &= \frac12 \int_{\bA^1(\Qbar)\times\bA^1(\Qbar)\setminus D} \int_Y G(z,w)\,d\mu(y)\,d\Delta(z)d\Delta(w)\\
   &= \frac12 \int_{\bA^1(\Qbar)\times\bA^1(\Qbar)\setminus D } m \int_Y (\rho_y,\rho_y)_y -(\rho_y, \delta_z)_y - (\rho_y, \delta_w)_y\,d\mu(y) \,d\Delta(z)d\Delta(w).
  \end{split}
 \end{equation}
Now again, we wish to show that we can add a nonnegative integrable function and make the integrand nonnegative, so we can apply Tonelli's theorem. By Lemma \ref{lemma:pairing-terms-bdd},
\[
 (\rho_y,\rho_y)_y \geq -2C(y),
\]
and, as we noted in the proof of Lemma \ref{lemma:integrability-of-delta-heights} above,
\[
 -(\rho_y-\lambda_y,\delta_z)_y = g_y(z).
\]
 By Lemma \ref{lemma:p-is-g-plus-log},
 $
  g_y(z) = p_{\rho_y}(z) - \log^+\,\abs{z}_y,
 $
 so the bound 
 \[
  \sup_{z\in\sA^1_y} \abs{g_y(z)} = C(y)
 \]
 yields that 
 \[
  -(\rho_y, \delta_z)_y \geq -C(y).
 \]
 Thus we can apply Tonelli's theorem and interchange the integrals once more:
 \begin{equation}\label{eqn:height-without-discr-term}
  \begin{split}
   h_\rho(\Delta) &= \frac{m}{2} \int_Y \int_{\sA^1_y\times \sA^1_y\setminus\mathrm{Diag}_y } (\rho_y,\rho_y)_y -(\rho_y, \delta_z)_y - (\rho_y, \delta_w)_y\,d\Delta(z)d\Delta(w)\,d\mu(y)\\
   &= \frac12 \int_Y (\rho_y,\rho_y) - 2(\rho_y, \Delta)_y\,d\mu(y). 
  \end{split}
 \end{equation}
 Using this, we see that 
 \begin{equation*}
  \begin{split}
    \sum_n t_n h_\rho(\Delta_n) &= \sum_n t_n \frac12 \int_Y (\rho_y,\rho_y) - 2 (\rho_y, \Delta_n)_y\,d\mu(y)\\
    &= \sum_n \frac12 \int_Y t_n (\rho_y,\rho_y) - 2t_n (\rho_y, \Delta_n)_y\,d\mu(y)\\
    &= \sum_n \frac12 \int_Y t_n (\rho_y,\rho_y) - 2 (\rho_y, t_n\Delta_n)_y\,d\mu(y).
  \end{split}
 \end{equation*}
 Again, using the same bounds as above,
 \[
  (\rho_y,\rho_y) - 2 (\rho_y, t_n\Delta_n)_y\geq -6C(y),
 \]
 so again, Tonelli's theorem applies and we can interchange the sums and use the linearity of the pairing to say that
\begin{equation*}
  \begin{split}
    \sum_n t_n h_\rho(\Delta_n) &= \sum_n \frac12 \int_Y t_n (\rho_y,\rho_y) - 2 (\rho_y, t_n\Delta_n)_y\,d\mu(y)\\
    &= \frac12 \int_Y \sum_n t_n (\rho_y,\rho_y) - 2 \sum_n (\rho_y, t_n\Delta_n)_y\,d\mu(y)\\
    &= \frac12 \int_Y (\rho_y,\rho_y) - 2 \bigg(\rho_y, \sum_n t_n\Delta_n\bigg)_y\,d\mu(y) = h_\rho\bigg(\sum_n t_n \Delta_n\bigg),
  \end{split}
 \end{equation*}
 which is what we wanted to prove.
\end{proof}

\subsection{Measures defined over a single number field}\label{sec:comparision-to-MY}
In this section we will compare our notion of a generalized adelic measure to the earlier notions of an \textit{adelic measure} in Favre and Rivera-Letelier \cite{FRL} and of a \textit{quasi-adelic measure} in \cite{MavrakiYe}. The main focus of this section is proving that, when our generalized adelic measure is defined over a single number field, the finiteness conditions given in Mavraki and Ye are in fact equivalent to those here. This section can be safely skipped if the reader is interested only in the new results of this paper.

First, observe that if $K$ is a number field and $\rho = (\rho_v)_{v\in M_K}$ is an adelic measure in the sense of Favre and Rivera-Letelier, then $\rho$ extends to a generalized adelic measure defined over $K$. Indeed, for a place $y \in Y$, if we let $v \in M_K$ be the place such that $y \mid v$, we can define $\rho_y = \rho_{v}$ and $g_y = g_{v}$. For places $v \in M_K$ lying above $p \in M_\bQ$, the functions $g_v(z) = g(v,z)$ are continuous on $\sP^1(\C_p)$, normalized so that $g_v(\infty) = 0$, and they satisfy $\Delta g_v = \rho_v - \lambda_v$, hence the same is true for $g_y$ for all $y \in Y(\Q,p)$. Moreover, since $\rho_v$ must be the standard measure for all but finitely many places $v \in M_K$, we have $C(v) := \sup_{z \in \sP^1(\C_v)} |g_v(z)| = 0$ for all but finitely many places $v \in M_K$, hence the integral $\int_Y C(y) d\mu(y)$ is really the finite sum $\sum_{v \in M_K} C(v)$, which yields a finite value. Thus, in the case that our generalized adelic measure is defined over a number field $K$ and only differs from the standard measure at a finite number of places, our measure can be identified with an adelic measure. 

Now suppose that our measure is defined over a single number field, but perhaps has an infinite number of places where it is not the standard measure. We wish to show that it defines a quasi-adelic measure in the sense of Mavraki and Ye. Let us recall their construction. Given measures $(\rho_v)_{v\in M_K}$, they define the \emph{inner} and \emph{outer radii} in the following fashion. First, they assume that each $\rho_v$ admits a continuous potential, that is, that there is a continuous function $g : \sP^1_v \ra \bR$ such that \[\Delta g = \rho_v - \lambda_v.\] (We note in passing that their normalized Laplacian has the opposite sign as ours, so this equation is the reverse of what appears in their paper.) As any two choices of $g$ differ only by a constant, $g$ is uniquely determined by its value at $\infty$. Define \[p_v(z) = g(z) + \log^+\,\abs{z}_v.\] They normalize so that if we define $G : \bC_v^2\ra \bR\cup\{-\infty\}$ by 
\[
 G(z,w) = \begin{cases}
  p_v(z/w) + \log\,\abs{w}_y &\text{if }z,w\in \bC_v\text{ and }w\neq 0,\\
  \log\,\abs{z}_v + g(\infty) &\text{if }z,w\in \bC_v\text{ and }z\neq 0,\ w = 0,\\
  -\infty &\text{if }z=w=0,
 \end{cases}
\]
then $G$ is a homogeneous logarithmic potential. (Note that the opposite sign in front of $g(\infty)$ versus equation (2.3) in \cite{MavrakiYe} is a result of our opposite choice of sign for the Laplacian.) In particular, $G(\alpha z,\alpha w) = \log\,\abs{\al}_v + G(z,w)$ for all $z,w\in\bC_v$ and $\alpha\in\bC_v^\times$. For any $z\in \bP^1(\bC_v)$, we let  $\twid z = (z_0,z_1)$ denote any homogeneous lift of $z$ to $\bC_v^2$, and for $z=(z_0,z_1)$ and $w=(w_0,w_1)\in \bC_v^2$, we let $z\wedge w = z_0 w_1 - z_1 w_0$. If we then define the Arakelov-Green's function  $g_{\rho_v} : \sP^1_v \times \sP^1_v \ra \bR\cup\{\infty\}$ by defining it for $z,w\in\bP^1(\bC_v)$ by
\[
 g_{\rho_v}(z,w) = - \log \,\abs{\twid{z}\wedge \twid{w}}_v + G(\twid z) + G(\twid w),
\]
then this extends naturally to a function on $\sP^1_v\times \sP^1_v$ which is continuous and finite off of the diagonal. (We refer the reader to \cite[Section 3.4]{BakerRumely} for more details on these Arakelov-Green's functions.) Notice that $G$, and hence $g_{\rho_v}$, are only defined up to the choice of constant $g(\infty)$. This is chosen so that the Arakelov-Green's function satisfies the normalization condition
\begin{equation}\label{eqn:a-g-is-normalized}
 \iint_{\sP^1_v\times \sP^1_v} g_{\rho_v}(z,w)\,d\rho_v(z)\,d\rho_v(w) = 0.
\end{equation}
 As observed in \cite[Section 3.5]{BakerRumely}, this normalization is in fact equivalent to the (seemingly stronger) condition that  \begin{equation}\label{eqn:BR-arakelov-greens-normalization}
 \int_{\sP^1_v} g_{\rho_v}(z,w)\,d\rho_v(z) \equiv 0.
 \end{equation}
 With this normalization, we then define the adelic set $(M_v)_{v\in M_K}$ by
 \[
 M_v = \{(z,w)\in \bC_{v}^2 : G(z,w)\leq 0\}.
 \]
 In fact, $M_v$ is a compact, circled, and pseudoconvex set in $\bC_v^2$, and by our choice of normalization, it has homogeneous capacity $\ON{cap}(M_v) = 1$. (The notion of homogeneous capacity was introduced by DeMarco \cite{DeMarcoDynamicsCapacity} in the archimedean setting, and extended to the non-archimedean setting by Baker and Rumely \cite[Section 3.3]{BakerRumely}.) Endow $\bC_v^2$ with the norm $\norm{(z,w)} = \max\{\abs{z}_v,\abs{w}_v\}$ and let $B(r) = \{(z,w)\in\bC_v^2 : \norm{(z,w)}\leq r\}$. Then the \emph{inner radius of $M_v$} is defined to be
 \[
  r_{\text{in},v} = \sup \{ r > 0 : B(r) \subset M_v \},
 \]
 and the outer radius of $M_v$ is defined to be 
 \[
  r_{\text{out},v} = \inf \{ r> 0 : M_v \subset B(r)\}.
 \]
 For a measure $\mu_v$ which admits a continuous potential we define the inner and outer radii of the measure to the inner and outer radii of the associated set $M_v\subset \bC_v^2$ constructed with the above normalization. 
 
 Then Mavraki and Ye define for a number field $K$ the sequence of measures $(\rho_v)_{v\in M_K}$ to be a \emph{quasi-adelic measure} if the following conditions are met:
 \begin{enumerate}
    \item \label{enum:my-1}Each $\rho_v$ is a probability measure on $\sP^1_v$ which admits a continuous potential with respect to the standard measure.
    \item \label{enum:my-2}We have the finiteness condition:
    \begin{equation}\label{eqn:r-in-r-out-converges}
     \sum_{v\in M_K} [K_v:\bQ_v]\cdot \abs{\log  r_{\text{in},v}} < \infty \quad\text{and} \quad 
     \sum_{v\in M_K} [K_v:\bQ_v]\cdot \abs{\log  r_{\text{out},v}} < \infty.
    \end{equation}
 \end{enumerate}
 If $\rho$ is a generalized adelic measure defined over $K$ and we identify it with $(\rho_v)_{v\in M_K}$, then clearly it meets condition \eqref{enum:my-1}. What must be shown is that our finiteness condition \eqref{enum:gen-adelic-2} from Definition~\ref{defn:gen-adelic-measure} is equivalent to the finiteness condition \eqref{enum:my-2} stated above for the inner and outer radii. It is not immediately obvious that this geometric condition on the adelic set $M_v$ is related to the bounds of the potential function $g(z)$, however, we will see that this is indeed the case, and while our normalization differs slightly, the two conditions are indeed equivalent.
 
 In order to see this, we first observe that with the notation above, for any $z\in\sP^1_v$ and $\twid{z}\in\bC_v^2$ a lift, $\log\, \norm{\twid{z}}_v - G(\twid{z})$ does not depend on the choice of lift, as each factor scales logarithmically with the absolute value of $\alpha$ if we replace $\twid{z}$ with $\alpha \twid{z}$. Thus in fact, for $z\in \bC_v$, 
 \[
  \log\, \norm{\twid{z}}_v - G(\twid{z}) = \log^+\abs{z}_v - p_v(z) = -g(z).
 \]
 Now, observe that for any $\twid{z}$, by scaling by some $\al\in \bC_v^\times$, we can arrange that $G(\al \twid{z}) = 0$, so $\twid{z}\in M_v$. By the definition of $r_{\text{in},v}$ and $r_{\text{out},v}$, it follows that 
 \[
  r_{\text{in},v}\leq \norm{\al \twid{z}} \leq r_{\text{out},v},
 \]
 so $-g(z) = \log\, \norm{\twid{z}}_v - G(\twid{z}) = \log\,\norm{\al\twid{z}}$, and thus 
 \[
   \log r_{\text{in},v} \leq -g(z) \leq \log r_{\text{out},v}.
 \]
  In fact, by continuity and the density of $\bC_v$ in $\sP^1_v$, we have
 \begin{equation}\label{eqn:r-bounds-on-g}
  \log r_{\text{in},v} = \inf_{z\in \sP^1_v} -g(z)
  \quad\text{and}\quad 
  \log r_{\text{out},v} = \sup_{z\in \sP^1_v} -g(z). 
 \end{equation}
 Thus we see that the inner and outer radii are in fact bounds on the potential function $g$ for which $\Delta g = \rho_v - \lambda_v$. However, Mavraki and Ye's choice of potential function differs from ours by a constant. We chose to normalize so that $g(\infty) = 0$, however, Mavraki and Ye normalized so that the double integral of the Arakelov-Green's function vanishes as in \eqref{eqn:a-g-is-normalized} above. To determine the value of $g(\infty)$ in their normalization, we use the fact that the function $F:\sP^1_y \ra\bR$ given by
 \[
 F(w) = \int_{\sP^1_v} g_{\rho_v}(z,w)\,d\rho_v(z) 
 \]
 is in fact constant as a function of $w$, as one can show that the function is harmonic on $\sP^1_y$; see the discussion in \cite[Section 3.5, Lemma 5.14]{BakerRumely} for more details. By the normalization in \eqref{eqn:BR-arakelov-greens-normalization}, $F\equiv 0$, and we will evaluate $F(\infty)$ in order to determine the constant:
 \begin{align*}
  F(\infty) &= \int_{\sP^1_v} G(\twid{z}) + G(1,0) - \log\, \abs{\twid{z}\wedge (1,0)}_v\,d\rho_v(z)\\
  &= \int_{\sP^1_y} p_v(z) + g(\infty) - \log\,\abs{0\cdot z - 1\cdot 1}_v \,d\rho_v(z)\\
  &= g(\infty) + \int_{\sP^1_v} p_v(z)\,d\rho_v(z).
 \end{align*}
 Now let
 \[
 u(z) = \int_{\sP^1_v}\log \,\abs{z-w}_{v}\,d\rho_v(w),
 \]
 and observe that, by considering asymptotics as $z\ra\infty$,
 $
 p_v(z) = g(\infty) + u(z).
 $
 Further, by \cite[Lemmes 2.5, 4.4]{FRL},
 \[
  \int_{\sP^1_v} u(z)\,d\rho_v(z) = -(\rho_v,\rho_v)_v.
 \]
 It follows that 
 \[
  F(\infty) = 2g(\infty) - (\rho_v,\rho_v)_v = 0,\quad\text{so}\quad
  g(\infty) = \frac12 (\rho_v,\rho_v)_v.
 \]
 Now let us define the analogue of our bounds on the potential function for this measure. Let $g_v : \sP^1_v \ra \bR$ be the potential function with $\Delta g_v = \rho_v - \lambda_v$ and normalized according to Mavraki and Ye's normalization $g_v(\infty) = \frac12 (\rho_v,\rho_v)_v$. Then if we define
 \[
  C_v = \sup_{z\in\sP^1_v} \abs{g_v(z) - g_v(\infty)},
 \]
 then for the generalized measure $(\rho_y)_{y\in Y}$ defined by $\rho_y = \rho_v$ where $y\mid v$, the bounding function $C : Y \ra [0,\infty]$ will be given by 
 \begin{equation}\label{eqn:C-y-in-terms-of-C-v}
  C(y) = \sum_{v\in M_K} C_v \chi_{Y(K,v)}(y).
 \end{equation}
 Then 
 \[
  \int_Y C(y)\,d\mu(y) = \sum_{v\in M_K} \frac{[K_v:\bQ_v]}{[K:\bQ]} C_v. 
 \]
  Now, we wish to show that $C\in L^1(Y)$ if and only if equation \eqref{eqn:r-in-r-out-converges} holds. First, suppose $C\in L^1(Y)$. Then 
  \[
   \sup_{z\in \sP^1_v} \abs{g(z)} \leq C_v + \abs{g(\infty)} = C_v + \frac12 \abs{(\rho_v,\rho_v)_v} \leq 3C_v,
  \]
  where we have applied Lemma \ref{lemma:pairing-terms-bdd} to bound $\abs{(\rho_v,\rho_v)_v}\leq 4C_v$. It follows that 
  \[
   -3C_v \leq \log r_{\text{in},v} \leq \log r_{\text{out},v} \leq 3C_v
  \]
  and so
  \[
   \sum_{v\in M_K} [K_v:\bQ] \abs{\log r_{\text{in},v}} \leq [K:\bQ] \cdot 3 \int_Y C(y)\,d\mu(y) < \infty,
  \]
  and likewise for the outer radius. 
  
  Now assume that $(\rho_v)_{v\in M_K}$ meets the condition \eqref{eqn:r-in-r-out-converges}. By equation \eqref{eqn:r-bounds-on-g},
  \[
   \sup_{z\in\sP^1_v} \abs{g_v(z)} = \max \{ \abs{\log r_{\text{in},v}}, \abs{\log r_{\text{out},v}}\}. 
  \]
  It follows that for our normalization,
  \[
   C_v = \sup_{z\in\sP^1_v} \abs{g_v(z) - g_v(\infty)} \leq 2 \sup_{z\in\sP^1_v} \abs{g_v(z)} = 2 \max \{ \abs{\log r_{\text{in},v}}, \abs{\log r_{\text{out},v}}\}.
  \]
  In particular, this now guarantees that for $C(y)$ defined as in \eqref{eqn:C-y-in-terms-of-C-v}, 
  \[
  \int_Y C(y)\,d\mu(y) = 2\sum_{v\in M_K} \frac{[K_v:\bQ_v]}{[K:\bQ]} \max \{ \abs{\log r_{\text{in},v}}, \abs{\log r_{\text{out},v}}\} < \infty.
  \]
  Thus, the two finiteness conditions are equivalent. 
  
\subsection{Equidistribution results}
In order to define the regularization techniques used in the following section, we start by recalling the Favre and Rivera-Letelier retraction-by-$\ep$ map on $\sP^1(\bC_p)$, for non-archimedean primes $p$. Let $\ep \geq 0$ be given, and we define $\pi_\ep(\zeta)$ to be the unique minimal point in $\sP^1_p$ on the arc between $\zeta$ and $\infty$ which has diameter at least $\ep$. We will want to vary the function in $\ep$, so we start by proving $\pi_\ep(\zeta)$ is continuous as a function of two variables:
\begin{lemma}\label{lemma:retraction-is-two-var-cts}
 Define the map
 \begin{align*}
  \pi : \bR_{\geq 0} \times \sP^1(\bC_p) &\ra \sP^1(\bC_p) \\
   (\ep, \zeta) &\mapsto \pi_\ep(\zeta).
 \end{align*}
 Then $\pi$ is a continous function, where $\sP^1(\bC_p)$ has the Berkovich topology, $\bR_{\geq 0}$ has the usual real subspace topology, and $\bR_{\geq 0} \times \sP^1(\bC_p)$ has the associated product topology.
\end{lemma}
\begin{proof}
 As in the proof of the continuity of $\pi_\ep$ in \cite[Lemme 4.7]{FRL}, we take advantage of the fact that a subbase for the Berkovich topology is given by sets of the form
 \[
  U(z,r) = \{ \zeta\in \sP^1(\bC_p) : \sup\{\zeta, z\} < r\}\quad{where}\quad z\in\bC_p,\ r > 0,
 \]
 and complements of closed discs:
 \[
  \sP^1(\bC_p) \setminus \overline{U(z,r)} = \{ \zeta\in \sP^1(\bC_p) : \sup\{\zeta, z\} > r\}\quad{where}\quad z\in\bC_p,\ r > 0.
 \]
 We start by proving that $\pi^{-1}(U(z,r))$ is open for every $z\in\bC_p$ and $r>0$. Notice that, by \cite[Lemme 4.7]{FRL}, $\pi_\ep^{-1}(U(z,r)) = U(z,r)$ if $\ep < r$, but $\pi_\ep^{-1}(U(z,r)) = \varnothing$ if $\ep \geq r$. It follows that
 \[
  \pi^{-1}(U(z,r)) = [0,r) \times U(z,r),
 \]
 which is clearly open in the product topology. On the other hand, if we let $B = \sP^1(\bC_p) \setminus \overline{U(z,r)}$ for some $z\in\bC_p$ and $r>0$, we see again by \cite[Lemme 4.7]{FRL} that $\pi_\ep^{-1}(B) = B$ if $0\leq \ep\leq r$ and $\pi_\ep^{-1}(B) = \sP^1(\bC_p)$ if $\ep>r$. It follows that 
 \[
  \pi^{-1}(B) = ([0,r]\times B) \cup ((r,\infty) \times \sP^1(\bC_p)),
 \]
 which is again an open set (note that it equals the union of two open sets, $[0,\infty)\times B$ and $(r,\infty)\times \sP^1(\bC_p)$).  It follows that $\pi$ is continuous.
\end{proof}

As we will use it in the next section, we introduce the notion of $\ep$-regularization of a measure.
\begin{defn}\label{defn:ep-reg-of-measure}
 Let $\rho$ be a measure on $\sP^1_y$ for a place $y$ of $\Qbar$ and let $\ep>0$ be given. If $y\mid \infty$, we define the $\ep$-regularization to be the measure $\rho_\ep$ given by requiring that for every continuous $f : \sP^1_y \ra \bR$,
 \[
  \int f\,d\rho_\ep = \int \left(\int_0^1 f(z + \ep e^{2\pi i t})\,dt\right) \,d\rho(z).
 \]
 If $y\nmid \infty$, then, following Favre and Rivera-Letelier \cite[\S 4.6]{FRL}, we define $\rho_\ep = (\pi_\ep)_*(\rho)$.
\end{defn}
\begin{rmk}
It will be clear to the reader that both of our regularizations are inspired by \cite{FRL}, however, the archimedean regularization differs in that we avoided convolving with a smooth function. As a result, our potential function will not be smooth but will still be continuous, which is all we require; further, when approximating discrete measures, we will avoid an extraneous term which appears in the estimates of \cite{FRL}. We refer the reader to the discussion in \cite[\S 2.1]{F-P-quant} for more details.

We also note that our regularization fails to smooth the potentials of point masses at $\infty\in\bP^1(\bC_y)$. Indeed, it is clear from the definition of our normalization that $(\delta_\infty)_\ep = \delta_\infty$. This was also true of the regularization in \cite{FRL}, but just as in their work, in our equidistribution results, the resulting measures will never charge a single classical point (that is, a point of $\bP^1(\bC_y)$), so we will be able to ignore the point at infinity.
\end{rmk}

We state the basic convergence lemma:
\begin{lemma}\label{lemma:convergence-of-regularizations}
 Let $\rho$ be a signed Borel measure on $\sP^1_y$. Then $\rho_\ep \ra \rho$ as $\ep\ra 0^+$ in the sense of weak convergence of measures. Further, if $\rho_n$ is a sequence of signed Borel measures such that $\rho_n \ra \rho$ weakly, then $(\rho_n)_\ep \ra \rho_\ep$ weakly as $\ep\ra 0^+$ as well.
\end{lemma}
\begin{proof}
 In the non-archimedean case, this is exactly \cite[Lemme 4.8]{FRL}, so we will not reprove it here. The archimedean case is, similarly to \cite[Lemme 2.7]{FRL}, trivial, but we will write it out for the sake of completeness. We identify $\sP^1_y$ with $\bP^1(\bC)$ and let a continuous real-valued function $f\in C(\bP^1(\bC))$ be given. We will show that 
 \[
  \int f \,d\rho_\ep \ra \int f\,d\rho.
 \]
 Note that by construction, 
 \[
  \int f \,d\rho_\ep = \int \left(\int_0^1 f(z + \ep e^{2\pi t})\,dt\right) \,d\rho(z)\ra \int f \,d\rho
 \]
 as $\ep\ra 0^+$, since the function $f$ is continuous on a compact space and therefore we can bring the limit inside the integral. Since this holds for each continous function, $\rho_\ep \ra \rho$ in the weak sense of measures. Now suppose that $\rho_n\ra \rho$ weakly, and fix $\ep>0$. We wish to show that $(\rho_n)_\ep \ra \rho_\ep$. Again, it suffices to show that for a given $f\in C(\bP^1(\bC))$,
 \[
  \int f \,d(\rho_n)_\ep \ra \int f \,d\rho_\ep \quad\text{as}\quad n \ra \infty. 
 \]
 But
 \[
  \int f \,d(\rho_n)_\ep = \int \left(\int_0^1 f(z + \ep e^{2\pi t})\,dt\right) \,d\rho_n(z) = \int f_\ep \, d\rho_n,
 \]
 where 
 \[
  f_\ep(z) = \int_0^1 f(z + \ep e^{2\pi t}).
 \]
 But $f_\ep$ is again a continuous function, so since $\rho_n\ra\rho$ weakly,
 \[
  \int f_\ep \, d\rho_n \ra \int f_\ep \, d\rho = \int f\,d\rho_\ep,
 \]
 where the last equality follows by the definition of $\rho_\ep$.
\end{proof}

\begin{lemma}\label{lemma:reg-has-cts-potential}
 Let $y$ be a place of $\Qbar$ and $\Delta$ be a discrete probability measure on $\sA^1_y$, given by
 \[
  \Delta = \sum_i t_i \delta_{\alpha_i},
 \]
 where $\sum_i t_i = 1$ and $\alpha_i\in \bC_y$, satisfying 
 \begin{equation}\label{eqn:local-finite-height}
  \sum_i t_i \log^+\,\abs{\alpha_i}_y < \infty.
 \end{equation}
  Then for any $0<\ep\leq 1$, the $\ep$-regularization $\Delta_\ep$ admits a continuous potential with respect to the standard measure $\lambda_y$, that is, there is a continuous function $g_y: \sP^1_y \ra \bR$ such that $g_y(\infty)=0$ and 
  \[\Delta g_y = \Delta_{\ep} - \lambda_y.\]
 \end{lemma}
 \noindent Notice that we have assumed that $\Delta$ does not charge $\infty\in \bP^1(\bC_y)$. This assumption is primarily for our convenience at the moment and we will see that it is not relevant in the proof of the equidistribution theorem, as measures of small height will not converge to measures that charge a single point like $\infty\in \bP^1(\bC_y)$. It is possible to introduce a regularization that allows for continuous potentials when $\Delta$ has support at $\infty$ as well, however, this would complicate our analysis and is unnecessary, so we do not do it here.
\begin{proof}
 First we recall for the reader's convenience that, for any $\ep>0$, the $\ep$-regularization of $\Delta$ is given by
 \[
  \Delta_{\ep} = \sum_{i} t_i \delta_{\alpha_i,\ep},
 \]
 where if $y\nmid\infty$, $\delta_{\alpha_i,\ep}=(\delta_{\alpha_i})_\ep$ denotes the point mass at $\pi_\ep(\alpha_i)=\zeta_{\alpha_n,\ep}$ and is the Favre-Rivera-Letelier smoothing of the measure, while if $y\mid \infty$, we let $\delta_{\alpha_i,\ep}=(\delta_{\alpha_i})_\ep$ denote the normalized arc length measure around the circle $\p D(\alpha_i,\ep)$ in $\bC$. We note that each individual measure $\delta_{\alpha_i,\ep}$ admits a continuous potential with respect to the standard measure $\lambda_y$ given by
 \[
  g_i(z) = \log\max\{\abs{z-\alpha_i}_y, \ep\} 
  - \log^+\, \abs{z}_y,
 \]
 where in the case of $y$ non-archimedean, we recall that $\abs{z-\alpha_i}_y$ is really the Hsia kernel $\delta_\infty(z,\alpha_i)$, or $\sup\{z,\alpha_i\}$ in the notation of Favre and Rivera-Letelier. Notice that
 \[
  \sup_{z\in\sP^1_y} \abs{g_i(z)} \leq \log(1 + \abs{\alpha_i}_y)\leq \log 2 + \log^+\,\abs{\alpha_i}_y.
 \]
 Define
 \[
  g(z) = \sum_i t_i g_i(z) = \sum_i t_i \log\max\{\abs{z-\alpha_i}_y, \ep\}
  - \log^+\, \abs{z}_y.
 \]
 Then
 \[
 \sum_i   t_i \abs{g_i(z)} \leq \log 2 + \sum_i t_i\log^+\,\abs{\alpha_i}_y < \infty
 \]
 by our assumption on $\Delta$. It follows by the Weierstrass $M$-test that $g\in C(\sP^1_y)$, as claimed.
\end{proof}

\begin{lemma}\label{lemma:lower-bound-on-Delta-n-height-pairing}
 Let $\rho$ be a generalized adelic measure and $\Delta$ be a discrete probability  measure on $\Qbar$ with finite height given by
 \[
  \Delta = \sum_{n=1}^\infty t_n \delta_{\alpha_n},
 \]
 where $\alpha_n\in \bA^1(\Qbar)$ are all distinct and $\sum_n t_n = 1$. Then for $\mu$-almost every $y\in Y$, there exists a continuous function $\eta_y : [0,\infty)\ra [0,\infty)$ such that $\eta(0)=0$, and that for all $\ep\in (0,1]$, the $\ep$-regularization $\Delta_\ep$ of $\Delta$ satisfies 
 \begin{equation}\label{eqn:Delta-pairing-vs-Delta-ep}
  (\rho_y - \Delta, \rho_y-\Delta)_y - (\rho_y - \Delta_\ep, \rho_y-\Delta_\ep)_y \geq -2 \eta_y(\ep) + \sum_{n=1}^\infty t_n^2 \log \ep
 \end{equation}
 and in particular,
  \begin{equation}\label{eqn:lower-bound-pairing-approx}
  (\rho_y - \Delta, \rho_y-\Delta)_y \geq -2 \eta_y(\ep) + \sum_{n=1}^\infty t_n^2 \log \ep
 \end{equation}
 for almost all $y\in Y(\bQ,p)$. Further, the functions $\eta_y$ depend on $y$ and on $\rho$, but not on $\Delta$.
\end{lemma}
\noindent We note that our lemma is inspired by Propositions 5 and 6 in Section 2 of \cite{F-P-quant}, which itself was based closely on Propositions 2.8 and 4.9 of \cite{FRL} but uses a slightly different regularization of the discrete measures at the archimedean places.
\begin{proof}
 Let $N = \{y\in Y : C(y)=\infty\}$, where $C(y)$ is the bounding function associated to our generalized adelic measure $\rho$ in Definition \ref{defn:gen-adelic-measure}. We have by assumption that $C(y)\in L^1(Y)$, so in particular, $\mu(N)=0$. We note that by the condition that $\Delta$ has finite height,
 \[
  \int_Y \int_{\sA^1_y} \log^+\,\abs{z}_y\,d\Delta(z) \,d\mu(y) < \infty,
 \]
 and it follows that for almost all $y\in Y$, 
 \[
  \int_{\sA^1} \log^+\,\abs{z}_y \,d\Delta(z) < \infty.
 \]
 This condition will be necessary to apply \ref{lemma:reg-has-cts-potential} and conclude that $\Delta_\ep$ admits a continuous potential function, which will be vital to our proof. 
 
 Let $Y'\subset Y$ denote the set of places $y$ where the above integral is finite and $C(y)<\infty$. Notice that $\mu(Y\setminus Y') = 0$. We will prove the result for all $y\in Y'$. Let $y\in Y'$ be given. For each $\ep\in (0,1]$, let $\Delta_\ep$ be the $\ep$-regularization of the discrete measure $\Delta$ as defined above. We write $\Delta_\ep$ explicitly as 
 \begin{equation}\label{eqn:epsilon-regularization-of-delta}
  \Delta_\ep = \sum_{n=1}^\infty t_n \delta_{\alpha_n,\ep},
 \end{equation}
 where, if $y\nmid\infty$, $\delta_{\alpha_n,\ep}$ denotes the point mass at $\pi_\ep(\alpha_n)=\zeta_{\alpha_n,\ep}$ and is the Favre and Rivera-Rivera smoothing of the measure, while if $y\mid \infty$, we let $\delta_{\alpha_n,\ep}$ denote the normalized arc length measure around the circle $\p D(\alpha_n,\ep)$ in $\bC$. (The restriction that $\ep\leq 1$ is only relevant in the non-archimedean setting.) We then proceed by writing
 \begin{equation}\label{eqn:main-ineq-for-delta-pairing-lb}
 \begin{split}
   &(\rho_y - \Delta, \rho_y-\Delta)_y 
  - (\rho_y - \Delta_\ep, \rho_y-\Delta_\ep)_y\\
  &\hspace{35mm} = -2(\rho_y, \Delta - \Delta_\ep)_y + (\Delta,\Delta)_y - (\Delta_\ep,\Delta_\ep)_y.
 \end{split}
 \end{equation}
 We now proceed by finding lower bounds for the terms on the right hand side. We start with the difference $(\Delta,\Delta)_y - (\Delta_\ep,\Delta_\ep)_y$, largely following the proofs of \cite[Prop. 5]{F-P-quant} and \cite[Lemme 4.11]{FRL}. We write
\[
 (\Delta,\Delta)_y = \sum_{i\neq j} -t_i t_j\log\,\abs{\alpha_i-\alpha_j}_y,
\]
and by the linearity of the energy pairing, 
\[
 (\Delta_\ep,\Delta_\ep)_y = \sum_{i,j} t_i t_j(\delta_{\alpha_i,\ep},\delta_{\alpha_j,\ep})_y. 
\]
 If $y\mid \infty$, then, following the idea proof of the proof of \cite[Lemme 2.10]{FRL}, then for $z\neq w\in\bC$,
\begin{align*}
 -(\delta_{z,\ep},\delta_{w,\ep}) &= \int_0^1\int_0^1 \log\abs{z + \ep\cdot e^{2\pi ir} - (w + \ep\cdot e^{2\pi is})}\,dr\,ds\\
 & = \int_0^1\max\{ \log\abs{z- (w + \ep\cdot e^{2\pi is})}, \log \ep\}\,ds\\
 &\geq \max\left\{  \int_0^1\log\abs{z- (w + \ep\cdot e^{2\pi is})}\,ds, \log \ep\right\}\\
 &\geq \max\{\log\abs{z-w},\log \ep\}\geq \log\abs{z-w} = -(\delta_z,\delta_{w}),
\end{align*}
 while if $y\nmid \infty$,
 \[
  \log\,\sup\{\zeta_{\alpha_i,\ep},\zeta_{\alpha_j,\ep}\} = \log\max\{ \ep, \abs{\alpha_i-\alpha_j}_y\} \geq \log \,\abs{\alpha_i-\alpha_j}_y,
 \]
 so that for all $i\neq j$ in the double sums,
 \[
  (\Delta,\Delta)_y - (\Delta_\ep,\Delta_\ep)_y \geq 0.
 \]
 Further, for the $i=j$ terms, $(\delta_{z,\ep},\delta_{z,\ep})_y$ is the Robin constant for the disc of radius $\ep$ for both non-archimedean and archimedean places, and in both cases, one obtains $(\delta_{z,\ep},\delta_{z,\ep})_y = -\log \ep$, so that
 \begin{equation}\label{eqn:7-1}
  (\Delta,\Delta)_y - (\Delta_\ep,\Delta_\ep)_y \geq \sum_{i=1}^\infty t_i^2 \log \ep.
 \end{equation}
 It remains to bound the term $-2(\rho_y, \Delta - \Delta_\ep)_y$. Our proof now breaks down into two cases. First, assume that $p\nmid \infty$ is a non-archimedean prime. Note that 
 \[
  (\rho_y, \Delta - \Delta_\ep)_y = \sum_{i=1}^\infty t_i \left(p_{\rho_y}(\alpha_i) - p_{\rho_y}(\zeta_{\alpha_i,\ep})\right),
 \]
 where $p_{\rho_y}(z)$ is the potential associated to $\rho_y$. By Lemma \ref{lemma:p-is-g-plus-log}, we know that $p_{\rho_y}(z) = g_y(z) + \log^+\,\abs{z}_p$. Now, since $g_y(z) = g(y,z)$ is assumed to be continuous as a map $Y(\bQ,p)\times \sP^1(\bC_p)\ra\bR$, we can define
 \[
  \eta_y(\ep) = \sup_{\zeta\in\sP^1_p} \abs{g_y(\pi_\ep(\zeta)) - g_y(\zeta)} , 
 \]
 and it follows by continuity of all of the variables involved and the compactness of the spaces that these suprema exist and are attained, that the resulting function is continuous in $\ep$, and, further, that $\eta(0) = 0$. Further, 
 \begin{align*}
  \abs{p_{\rho_y}(\alpha_i) - p_{\rho_y}(\zeta_{\alpha_i,\ep})} &\leq 
  \abs{g_y(\alpha_i) - g_y(\zeta_{\alpha_i,\ep})} + \abs{\log^+ \abs{\alpha_i}_p - \log^+\,\abs{\zeta_{\alpha_i,\ep}}}\\
    &\leq \eta_y(\ep),
 \end{align*}
 since $\log^+\,\abs{\zeta_{z, \ep}}_p = \log^+\,\abs{z}_p$ whenever $\ep \leq 1$. It follows that 
\begin{equation}\label{eqn:7-2}
 \abs{(\rho_y, \Delta - \Delta_\ep)_y} \leq \sum_{i=1}^\infty t_i \eta_y(\ep) = \eta_y(\ep).
\end{equation}
 
 For the second case, suppose that $y\mid \infty$. Let $\hat\eta_y : [0,\infty)\ra [0,\infty)$ be the function given by
 \[
  \hat\eta_y(\ep) =\sup_{\substack {z,w \in\bP^1(\bC)\\ \sigma(z,w)\leq r} } \abs{g_y(z) - g_y(w)} ,
 \]
 where $\sigma$ denotes the spherical metric on $\bP^1(\bC)$. Again, by the continuity of $g_y$ and compactness of the space $\sP^1_y$, $\hat{\eta}_y(\ep)$ is a continuous function and works as a uniform modulus of continuity for the potential functions over all places $y\mid \infty$. Let $\eta_y(\ep) = \hat\eta_y(\ep) + \ep$. Then for $z\in \bC$, 
  \begin{align*}
  \bigg|p_{\rho_y}(z) - \int p_{\rho_y}d\delta_{z,\ep}\bigg| &\leq 
  \bigg|g_y(z) - \int g_y d\delta_{z,\ep}\bigg| + \bigg|\int \log^+\,\abs{w}\,d\delta_{z,\ep}(w) - \log^+\,\abs{z}\bigg|\\
    &\leq \hat \eta_y(\ep)  + \ep = \eta_y(\ep),
 \end{align*}
where we have used  
 \[
 \bigg|\int_0^1 \log^+\,\abs{z + \ep \cdot e^{2\pi i t}}\,dt - \log^+\,\abs{z}\bigg| \leq \ep 
 \]
 to bound the second term, from which we again derive \eqref{eqn:7-2}, this time in the archimedean case. Combining equation \eqref{eqn:main-ineq-for-delta-pairing-lb} with \eqref{eqn:7-1} and \eqref{eqn:7-2} gives the first part of the conclusion, equation \eqref{eqn:Delta-pairing-vs-Delta-ep} above. 
 
 For the second part of the conclusion,  we note that by \cite[Prop. 2.6, Prop. 4.5]{FRL}, as $\Delta_\ep$ admits a continuous potential by Lemma \ref{lemma:reg-has-cts-potential}, we have
 \[
  (\rho_y - \Delta_\ep, \rho_y-\Delta_\ep)_y \geq 0.
 \]
 Therefore \eqref{eqn:main-ineq-for-delta-pairing-lb} yields
 \begin{equation}\label{eqn:lower-bound-on-pairing}
 (\rho_y - \Delta, \rho_y-\Delta)_y 
 \geq -2(\rho_y, \Delta - \Delta_\ep)_y + (\Delta,\Delta)_y - (\Delta_\ep,\Delta_\ep)_y,
 \end{equation}
 which, combined with \eqref{eqn:7-1} and \eqref{eqn:7-2} gives the second conclusion \eqref{eqn:lower-bound-pairing-approx}. 
\end{proof}

The previous lemma defines a condition for the local pairings to be nonnegative in the limit, to which we will give a name:
\begin{defn}\label{defn:well-distributed}
 A sequence of discrete probability measures $(\Delta_n)_{n\in\bN}$ on a topological space is \emph{well-distributed} if 
  \begin{equation}\label{eqn:Delta-n-well-distr-defn}
  \sum_{z\,\in \, \supp(\Delta_n)} \Delta_n(\{z\})^2 \ra 0\quad\text{as}\quad n\ra\infty.
 \end{equation}
  Notice that if the $\Delta_n$ are well-distributed, then we must have $\abs{\supp(\Delta_n)}\ra \infty$, and further, if the $\Delta_n$ are measures on $\Qbar$, this condition is independent of the particular embedding into $\sP^1_y$ for any place $y$ of $\Qbar$. 
\end{defn}
\begin{example}
Let $K/\bQ$ be a field, and for $n \in \N$ let $\alpha_n \in \Kbar$ and let $\Delta_n = [\al_n]_K$ be the probability measure equally supported on each $\Gal(\overline{K}/K)$-conjugate of $\al_n$.  If $[K(\al_n):K] \ra \infty$, then the sequence $(\Delta_n)_{n\in\bN}$ is well-distributed, as
\[
\sum_{z\,\in \, \supp(\Delta_n)} \Delta_n(\{z\})^2 = \sum_{i=1}^{[K(\alpha_n):K]} \frac{1}{[K(\alpha_n):K]^2} = \frac{1}{[K(\alpha_n):K]} \ra 0\quad\text{as}\quad n\ra\infty.
\]

 \end{example}

We will not use the definition of the well-distributed property directly in our proof of the equidistribution results, but rather we will use it in an equivalent but slightly more technical form:
\begin{lemma}\label{lemma:w-d-in-useful-form}
 If a sequence of discrete probability measures $(\Delta_n)_{n\in\bN}$ on a topological space is well-distributed, then there exists a sequence of real numbers $(\ep_n)_{n\in\bN}$ such that
 \[
  0 < \ep_n \leq 1\quad\text{for each}\quad n\in\bN, \quad 
  \ep_n \ra 0 \quad\text{as}\quad n\ra\infty,
 \]
and 
  \begin{equation}\label{eqn:Delta-n-condition}
  \sum_{z\,\in \, \supp(\Delta_n)} \Delta_n(\{z\})^2 \log \ep_n \ra 0\quad\text{as}\quad n\ra\infty.
 \end{equation}
\end{lemma}
\begin{proof}
 For each $n\geq 1$ let
 \[
  x_n = \sum_{z\,\in \, \supp(\Delta_n)} \Delta_n(\{z\})^2.
 \]
 The well-distributed condition says that $x_n\ra 0$ as $n\ra\infty$. Note that as $\Delta_n$ is a discrete probability measure, we must have $0<x_n\leq 1$. So we may simply take $\ep_n = x_n$, then
 \[
 \sum_{z\,\in \, \supp(\Delta_n)} \Delta_n(\{z\})^2 \log \ep_n = x_n \log x_n \ra 0\quad\text{as}\quad n\ra\infty.\qedhere
 \]
\end{proof}

 \begin{prop}\label{prop:AZ-implies-weak-conv}
 Let $y$ be a place of $\Qbar$, let $\rho_n$ be a sequence of Borel probability measures on $\sP^1_y$ each of which admits a continuous potential with respect to the standard measure $\lambda_y$, and let $\rho$ be a probability measure meeting the same conditions. If
 \[
  (\rho_n - \rho, \rho_n - \rho)_y \ra 0\quad\text{as}\quad n\ra\infty,
 \]
 then $\rho_n \ra \rho$ as $n\ra\infty$ in the sense of weak convergence of measures.
\end{prop}
\begin{proof}
 Let $g_\rho : \sP^1_y \times \sP^1_y \ra\bR\cup \{\infty\}$ be the function given by 
 \begin{equation}
  g_\rho(z,w) = p_\rho(z) + p_\rho(w) - \log\,\abs{z-w}_y + (\rho,\rho)_y,
 \end{equation}
 where $p_\rho(z)$ is the potential function of $\rho$, as defined in Lemma \ref{lemma:p-is-g-plus-log} above. We note that $g_\rho(z,w)$ is the Arakelov-Zhang function defined in \cite{BakerRumely}. It is easy to see that $g_\rho(z,w)$ is continous off of the diagonal, and lower-semicontinuous on the diagonal, where it is equals $\infty$. It follows from \cite[Lemmes 2.4, 2.5, 4.3, 4.4]{FRL} that we can ignore the diagonal (of classical points) in integrating $g_\rho(z,w)$ against $\rho_n$, so
 \[
  (\rho_n - \rho, \rho_n - \rho)_y = \iint_{\sP^1_y\times \sP^1_y} g_\rho(z,w)\,d\rho_n(z)\,d\rho_n(w).
 \]
 It follows from \cite[Prop. 2.6 and 4.5]{FRL} and the slightly stronger \cite[Theorem 3.45]{BakerRumely} that this pairing is nonnegative and vanishes if and only if $\rho_n = \rho$. (The theorem in \cite{BakerRumely} is stronger in that it removes the hypothesis that $\rho_n$ also admits a continuous potential.) 
 
 Let $\rho'$ be a weak limit of some subsequence of $\rho_n$, which must exist by Prokhorov's theorem. Since our space $\sP^1_y$ is metrizable (we refer the reader to the discussion of \cite[Chapter 1.5]{BakerRumelyBook} for details), we can apply the Portmanteau theorem for weak convergence of measures to the lower-semicontinuous function $g_\rho(z,w)$ to conclude that
 \begin{equation}
  \begin{split}
  0 = \liminf_{n\ra \infty} \iint_{\sP^1_y\times \sP^1_y} g_\rho(z,w)\,d\rho_n(z)\,d\rho_n(w) &\geq \iint_{\sP^1_y\times \sP^1_y} g_\rho(z,w)\,d\rho'(z)\,d\rho'(w)\\
  &= (\rho - \rho', \rho - \rho')_y \geq 0,
  \end{split}
 \end{equation}
 where the nonnegativity follows from the fact that our measures both have continuous potentials. But then $(\rho - \rho', \rho - \rho')_y = 0$ so $\rho = \rho'$, and since every subsequence of $\rho_n$ has a further subsequence which converges to this weak limit, the entire sequence has weakly limits to $\rho$. 
 \end{proof}

 We are now ready to prove the local equidistribution result. We note that our result is very much inspired by Propositions 2.11 and 4.12 of \cite{FRL}, however, there are several differences, most notably, that our discrete measures may in fact be infinitely supported.
\begin{prop}[Local equidistribution]\label{prop:local-equi}
 Let $y$ be a place of $\Qbar$ and $\rho$ be a Borel probability measure on $\sP^1_y$ which admits a continuous potential in the sense that there there exists a continuous function $g : \sP^1_y \ra \bR$ such that $g(\infty)=0$ and $\Delta g = \rho - \lambda_y$. Suppose that $\Delta_n$ is a sequence of discrete probability measures on $\sP^1_y$ which are well-distributed, which satisfy the local finiteness condition
 \begin{equation}\label{eqn:prop-l-e-finite-height}
  \int_{\sA^1_y} \log^+\, \abs{z}_y \,d\Delta_n(z) < \infty
 \end{equation}
  for each $n\in\bN$, and for which 
 \begin{equation}
  \limsup_{n\ra\infty}\ (\rho - \Delta_n, \rho - \Delta_n)_y \leq 0.
 \end{equation}
 Then the measures $\Delta_n$ converge weakly to $\rho$ in the sense of weak convergence of measures.
\end{prop}
\begin{proof}
  First, we note that we may assume that no $\Delta_n = \delta_\infty$, as the well-distributed condition requires that the support of $\Delta_n$ must grow, so we would be removing at most finitely many terms. Second, let us show that we may assume that no $\Delta_n$ measure charges $\infty\in\bP^1(\Qbar)$. To see this, note that if $\Delta_n(\{\infty\})\not\rightarrow 0$, then there exists a subsequence $\Delta_{n_k}$ with $\Delta_{n_k}(\{\infty\}) \geq c > 0$ for all $k\in\bN$.  But then the $\Delta_n$ are no longer well-distributed, as for these indices $n_k$ we have
    \[
    \sum_{z\,\in \, \supp(\Delta_{n_k})} \Delta(\{z\})^2 \ge c^2 > 0.
    \]
 Define a new sequence of probability measures by 
 \[
  \Delta_n' = \frac{1}{1-\Delta_n(\{\infty\})} \Delta_n|_{\sA^1_y}.
 \]
 Notice that, since $\Delta_n(\{\infty\})\ra 0$, any weak limit of any subsequence of $\Delta_n'$ is the same as the weak limit of any subsequence of $\Delta_n$. Therefore, by replacing $\Delta_n$ by $\Delta_n'$ if necessary, we can assume that no $\Delta_n$ measure charges $\infty\in \bP^1(\bC_y)$.
 
 Write each $\Delta_n$ as 
 \[
  \Delta_n = \sum_i t_i^{(n)} \delta_{z^{(n)}_{i}},
 \]
 where $z^{(n)}_{i}\in \bA^1(\bC_y)$ and $\sum_i t_i^{(n)} = 1$. By equation \eqref{eqn:Delta-pairing-vs-Delta-ep} from Lemma \ref{lemma:lower-bound-on-Delta-n-height-pairing}, there exists a function $\eta=\eta_y$, depending on $\rho$ and the place $y$, but not on $\Delta_n$, such that 
 \[
  (\rho_y - \Delta_n, \rho_y-\Delta_n)_y - (\rho_y - (\Delta_n)_\ep, \rho_y-(\Delta_n)_\ep )_y\geq -2 \eta(\ep) + \sum_{i=1}^\infty (t_i^{(n)})^2 \log \ep,
 \]
 or 
 \[
  (\rho_y - (\Delta_n)_\ep, \rho_y-(\Delta_n)_\ep )_y \leq (\rho_y - \Delta_n, \rho_y-\Delta_n)_y + 2 \eta(\ep) - \sum_{i=1}^\infty (t_i^{(n)})^2 \log \ep.
 \]
 Now by Lemma \ref{lemma:reg-has-cts-potential}, $(\Delta_n)_\ep$ admits a continuous potential for each $n$ and $\ep\in (0,1]$, so we also know that 
 \[
  0 \leq (\rho_y - (\Delta_n)_\ep, \rho_y-(\Delta_n)_\ep )_y.
 \]
Let $(\ep_n)$ be a sequence of real numbers in $(0,1]$ for which the conclusion of Lemma~\ref{lemma:w-d-in-useful-form} holds. For this sequence,
 \[
  2 \eta(\ep_n) - \sum_{i=1}^\infty (t_i^{(n)})^2 \log \ep_n \ra 0
 \]
as $n \to \infty$. It follows that 
\[
 \lim_{n\ra\infty} \ (\rho_y - (\Delta_n)_{\ep_n}, \rho_y-(\Delta_n)_{\ep_n} )_y = 0. 
\]
Since our measures admit continuous potentials, it follows from Proposition \ref{prop:AZ-implies-weak-conv} that we have the weak convergence of measures 
\[
 \lim_{n\ra\infty} (\Delta_n)_{\ep_n} = \rho_y.
\]
Now, for any subsequence of the $\Delta_n$, we know by Prokhorov's theorem that there is a further subsequence, which we denote $(\Delta_{n_k})_{k=1}^\infty$, that has a weak limit, say $\rho'$. Then by Lemma \ref{lemma:convergence-of-regularizations}, for any $\ep_m$ 
we have
\[
 \lim_{k\ra\infty} (\Delta_{n_k})_{\ep_m} = \rho'_{\ep_m},
\]
which implies
\[
 \lim_{m\ra\infty} \lim_{k\ra\infty} (\Delta_{n_k})_{\ep_m} = \lim_{m\ra\infty} \rho'_{\ep_m} = \rho'.
\]
Since the subsequence with $m=n_k$ converges to $\rho$, so we must have $\rho'=\rho$, and it follows that every convergent subsequence of $\Delta_n$ converges to $\rho$, so $\Delta_n\ra \rho$, as claimed.
\end{proof}

\begin{thm}[Equidistribution for generalized adelic heights]\label{thm:gen-adelic-equi}
 Let $\Delta_n$ be a well-distributed sequence of probability measures on $\Qbar$ and $\rho$ be a generalized adelic measure. If $h_\rho(\Delta_n) \ra 0$ as $n\ra\infty$, then $\Delta_n \ra \rho_y$ in the weak sense of convergence of measures at $\mu$-almost every place $y\in Y$.
\end{thm}
\begin{rmk}
We note that if our generalized adelic measure is in fact defined over a single number field, then the condition of $\mu$-almost every place becomes every place of that number field, as all places of any number field have positive $\mu$ measure. We also remind the reader that weak convergence of measures here refers to weak-* convergence of the measures as positive linear operators on the space of (compactly supported) continuous functions of the Berkovich line $\sP^1_y$. 
\end{rmk}
\begin{proof}
 Let $f_n : Y \ra \bR$ be the functions given by the local energy pairings:
 \[
  f_n(y) = \frac12 (\rho_y - \Delta_n, \rho_y-\Delta_n)_y,
 \]
 so that  
 \[
  h_\rho(\Delta_n) = \int_Y f_n(y)\,d\mu(y). 
 \]
 Note that since $h_\rho(\Delta_n)\ra 0$, we may assume without loss of generality by tossing out at most finitely many terms of the sequence that $h_\rho(\Delta_n)<\infty$ for all $n$, hence that $\Delta_n$ meets the finite height condition by Corollary~\ref{cor:finite-height-is-same-for-all-heights}. Our first goal is to prove that $f_n(y) \ra 0$ for almost all places $y\in Y$. We will then show that $\Delta_n$ meets the local finiteness condition needed in Proposition \ref{prop:local-equi} for almost all $y\in Y$, so at the intersection of these two sets of places, we can apply our local equidistribution result Proposition \ref{prop:local-equi} to get the desired result. 
 
 In order to prove that $f_n(y)\ra 0$ for almost all $y$, we start by breaking $f_n(y)$ into its positive and negative parts, that is, we let
 \[
 f_n(y) = f_n^+(y) - f_n^-(y),
 \]
 where $f_n^+,f_n^- \geq 0$. Write
 \[
  \Delta_n = \sum_i t_i^{(n)} \delta_{\alpha_i^{(n)}},
 \]
 where the $\alpha_i^{(n)}\in \Qbar$, $t_i^{(n)}\geq 0$, and $\sum_i t_i^{(n)} = 1$. (Note that the sum for $\Delta_n$ is countably infinite or finite.) Applying Lemma \ref{lemma:lower-bound-on-Delta-n-height-pairing}, we know that for $\mu$-almost every place $y$, there is a continuous nonnegative function $\eta_y : [0,\infty)\ra [0,\infty)$ depending on $\rho$ and the $y$ such that $\eta_y(0)=0$ and such that 
 \[
  f_n^-(y) \leq \eta_y(\ep) + \frac12 \sum_{i} t_{n,i}^2 \log \ep^{-1} \quad\text{for all}\quad \ep\in(0,1].
 \]

 By Lemma \ref{lemma:w-d-in-useful-form} and our assumption that the $\Delta_n$ are well-distributed, there exists a sequence of choices$\ep = \ep_n$ with $0<\ep_n\leq 1$, $\ep_n \ra 0$ as $n\ra\infty$, and such that the right hand side of the above equation must vanish as $n\ra\infty$, where the first term vanishes by the continuity of $\eta_y$ and $\eta_y(0)=0$, and the second vanishes by \eqref{eqn:Delta-n-condition}. It follows that 
 $
  f_n^-(y) \ra 0 
 $ pointwise for $\mu$-almost every $y\in Y$.
 
 Our goal is now to prove that in fact $f_n^-$ is bounded by an integrable function, so that we can apply dominated convergence and conclude that $\int f_n^-(y)\,d\mu(y)\ra 0$. We start by writing:
 \[
  f_n(y) - \frac12 (\lambda_y - \Delta_n, \lambda_y - \Delta_n)_y = \frac12 (\rho_y,\rho_y)_y - (\rho_y - \lambda_y, \Delta_n)_y.
 \]
 If $y\not\in Y(\bQ,\infty)$, then by Lemma \ref{lemma:lower-bds-on-standard-pairings}, we have $(\lambda_y - \Delta_n, \lambda_y - \Delta_n)_y\geq 0$, so in fact,
 \[
  f_n(y) \geq \frac12 (\rho_y,\rho_y)_y - (\rho_y - \lambda_y, \Delta_n)_y.
 \]
 We recall that by Lemma \ref{lemma:pairing-terms-bdd},
 \[
  (\rho_y,\rho_y)_y \geq -2C(y),
 \]
 and
 \[
  - (\rho_y - \lambda_y, \Delta_n)_y = \int_{\sA^1_y} g_y(z) \,d\Delta_n(z),
 \]
 so $-(\rho_y - \lambda_y, \Delta_n)_y \geq - C(y)$. It follows that for $y\not\in Y(\bQ,\infty)$,
 \[
  f_n(y) \geq -3C(y),\quad\text{or}\quad f_n^-(y) \leq 3C(y).
 \]
 For $y\in Y(\bQ,\infty)$, let $\eta_y$ be the function from Lemma \ref{lemma:lower-bound-on-Delta-n-height-pairing} where it exists. Recall from the proof of Lemma \ref{lemma:lower-bound-on-Delta-n-height-pairing} that for $y\mid \infty$, $\eta_y(\ep) = \hat \eta_y(\ep) + \ep$ where
  \[
  \hat\eta_y(\ep) =\sup_{\substack {z,w \in\bP^1(\bC)\\ \sigma(z,w)\leq r} } \abs{g_y(z) - g_y(w)} .
 \]
 Now observe that as $C(y)= \sup_{z\in\sP^1_y} \abs{g_y(z)}$, we have $\hat \eta_y(\ep) \leq 2C(y)$ for all $y$ away from a set of $\mu$-measure zero. Then Lemma \ref{lemma:lower-bound-on-Delta-n-height-pairing} gives $f_n(y)\geq -2 \eta_\infty(\ep)$, that is, taking $\ep=1$,
 \[
 f_n^-(y)\leq 2\eta_y(1) \leq 2 C(y) + 1 \quad\text{for $\mu$-a.e.}\quad y\in Y(\bQ,\infty). 
 \]
 Define the function $F : Y \ra [0,\infty)$ by
 \[
 F(y) = \begin{cases}
         2C(y) + 1 &\text{if }y\in Y(\bQ,\infty),\\
         3C(y) &\text{if }y\in Y \setminus Y(\bQ,\infty).\\
        \end{cases}
 \]
 Clearly, $\int_Y F\,d\mu \leq 1+ 3\int C\,d\mu<\infty$, so $F$ is integrable. Lebesgue's dominated convergence theorem then applies and tells us that $\int_Y f_n^- \,d\mu \ra 0$, which, together with the assumption that 
 \[
  h_\rho(\Delta_n) = \int_Y f_n \,d\mu = \int_Y f_n^+ \,d\mu - \int_Y f_n^- \,d\mu\ra 0,
 \]
 tells us that $\int f_n^+ \,d\mu\ra 0$ as well. If follows that $\int \abs{f_n} \,d\mu \ra 0$, so we conclude that $f_n(y) \ra 0$ for $\mu$-almost all places $y\in Y$. 
 
 Now we show that for almost all $y\in Y$, we have that $\Delta_n$ meets the two conditions of Proposition \ref{prop:local-equi}, so we can conclude that $\Delta_n\ra \rho$ for $\mu$-a.e. place $y\in Y$. To see this, notice that as each $\Delta_n$ has finite height, 
 \[
  \int_Y \int_{\sA^1_y} \log^+\,\abs{z}_y \,d\Delta_n(z)\,d\mu(y) < \infty. 
 \]
 It follows that for a set of $\mu$-almost all places which we denote as $Y_n$,
 \[
  \int_{\sA^1_y} \log^+\,\abs{z}_y \,d\Delta_n(z)<\infty \quad\text{for all}\quad y\in Y_n.
 \]
 Since $\mu(Y\setminus Y_n) = 0$, we form the set
 \[
  Y' = \bigcap_{n=1}^\infty Y_n,
 \]
 and we can conclude that $\mu(Y\setminus Y') = 0$ as well. Let $Y'' = \{ y\in Y : f_n(y) \ra 0\}$. As we argued above, $\mu(Y\setminus Y'') = 0$, so if we let $Z = Y'\cap Y''$, then again, $\mu(Y\setminus Z) = 0$. For all $y\in Z$, $\Delta_n$ meets the conditions of Proposition \ref{prop:local-equi}, and so we can conclude that the set of $y$ for which we get the desired equidistribution result has full $\mu$-measure.
\end{proof}

\section{Stochastic Dynamics}\label{sec:stoch-dyn}

\subsection{Stochastic dynamical heights}
The main goal of this section is to establish that the stochastic height is in fact the height associated to a certain generalized adelic measure which depends on the family of maps $S$ and the probability measure $\v_1$ associated to this family. 

We will recall here the main results of Healey and Hindes \cite{HealeyHindes} in the context of arithmetic dynamics. Let $S$ be a finite or countably infinite set of rational maps defined over an algebraic closure $\Qbar$, with each map being of degree at least $2$. We view these as functions of the projective line over $\Qbar$ to itself. Let $\nu_1$ be a probability measure on $S$. For $n\in\bN$, let $\nu_n$ be the product measure induced by $\nu_1$ on $S^n$; that is, for $\gamma_n = (\varphi_1,\ldots,\varphi_n) \in S^n$, we have $\nu_n(\gamma_n) = \nu_1(\varphi_1) \cdots \nu_1(\varphi_n)$. For convenience, for each $\gamma_n=(\varphi_1,\ldots, \varphi_n)\in S^n$ we will write $\gamma_n(\al)$ for the composition $\varphi_n\circ \cdots\circ \varphi_1(\al)$. Furthermore, we let $\nu$ denote the unique product measure on $S^\bN$ which satisfies, for all $k\in\bN$ and $A_1,\ldots, A_k \subset S$, $\nu(A_1\times A_2\times \cdots \times A_k \times S\times S\times \cdots ) = \prod_{i=1}^k \nu_1(A_i)$ defined on the smallest $\sigma$-algebra generated by sets of the form $A_1\times \cdots A_k\times S\times S\times \cdots$. With this $\sigma$-algebra and measure we view $S^\bN$ as the space of all independent identically distributed (i.i.d.) sequences in $S$.  Given a sequence $\gamma = (\varphi_i)_{i=1}^\infty \in S^\bN$ and positive integer $n$, we will denote by $\gamma_n$ the composition $\gamma_n = \varphi_n \circ \ldots \circ \varphi_1$, and we will denote the degree of this composite map as 
\[
 \deg(\gamma_n) = \prod_{i=1}^n \deg(\varphi_i).
\]

As we will need to exclude these points from our equidistribution theorem, we define what it means for a point to be \emph{exceptional} for the stochastic system $S$:
 \begin{defn}\label{defn:exceptional-set}
  Let $S$ be a countable set of rational maps defined over $\Qbar$. We define the \emph{exceptional set} of the system $S$ to be the set $E_{S}$ of all points $\alpha \in \bP^1(\Qbar)$ such that the grand orbit
    \[
        \cO^{\pm}_{S} (\alpha) := \bigcup_{n=0}^\infty \bigcup_{\gamma_n \in S^n}\left(\gamma_n^{-1}(\alpha) \cup \{\gamma_n(\alpha)\}\right)
    \]
 is finite.
\end{defn}
\begin{rmk}\label{rem:exceptional}\mbox{}
\begin{enumerate}
    \item While in classical dynamics the exceptional set is traditionally defined in terms of finiteness of the grand orbit, it is equivalent to require only that the backward orbit
    \[
        \cO^-_S (\alpha) := \bigcup_{n=0}^\infty \bigcup_{\gamma_n \in S^n}\gamma_n^{-1}(\alpha)
    \]
    is finite. Indeed, any finite set which is backward invariant under a collection of surjective maps is necessarily forward invariant as well.
    \item It is straightforward to see that the exceptional set for $S$ must be contained in the exceptional set for every map $\varphi \in S$. In particular, this implies that $|E_S| \le 2$.
    \item On the other hand, a point $\alpha$ may be exceptional for every $\varphi\in S$ without being exceptional for $S$. For example, let $\varphi_1(z) = \frac{1}{z^2}$ and $\varphi_2(z) = z^2 + 1$, and take $S = \{\varphi_1,\varphi_2\}$. Then $\infty$ is exceptional for both maps $\varphi_1$ and $\varphi_2$, but $\infty$ is not exceptional for $S$, since $\varphi_1^{-1}(\infty) = \{0\}$ and $0$ is not exceptional for $\varphi_2$.
\end{enumerate}
\end{rmk}

Let $y\in Y$ be a place of $\Qbar$. For each $\varphi \in S$ we can find a continuous function $g_{\varphi,y} : \bP^1(\bC_v) \ra \bR$ such that $g_{\varphi,y}$ is continuous and 
\begin{equation}\label{eqn:first-g-phi-y}
 \Delta g_{\varphi,y} = \frac{1}{\deg \varphi} \varphi^*(\lambda_y) - \lambda_y,
\end{equation}
where $\lambda_y$ denotes the equilibrium measure of the unit disc in $\sA_y^1$, which we remind the reader is the normalized Lebesgue measure of the unit circle in $\bC$ if $y$ is archimedean, or the unit point mass at the Gauss point if $y$ is non-archimedean. Such a $g_{\varphi,y}$ is not difficult to construct; for example, if we write $\varphi(z) = f(z)/g(z)$ where $f,g\in \Qbar[z]$ are polynomials with no common factors, then
\[
 g_{\varphi,y}(z) = \frac{1}{\deg \varphi}\log\max \{ \abs{f(z)}_y, \abs{g(z)}_y\} - \log^+\,\abs{z}_y
\]
suffices, with the natural extension to the Berkovich line in the non-archimedean setting. Note that, depending on our choice of representation of $\varphi(z)$ as a quotient, our $g(z)$ functions may differ by a constant. We will choose to normalize our $g_{\varphi,y}(z)$ functions so that $g_{\varphi,y}(\infty) = 0$. In order to construct the canonical measure associated to $\varphi$ at $v$, one ordinarily forms the telescoping series:
\[
 G_{\varphi,y}(z) = \sum_{n=0}^\infty \frac{1}{(\deg \varphi)^n} g_{\varphi,y}(\varphi^n(z)).
\]
It follows that $G_{\varphi,y}(z)$ is also continuous on $\sP_y^1$ and one can show (see, for example, \cite[Th\'eor\`eme 8]{FRL}) that $\Delta G_{\varphi,y} = \mu_{\varphi,y} - \lambda_y$. 

Each $g_{\varphi,y}(z)$ is continuous and bounded on $\sP_y^1$, but in order for our resulting system to be given by a generalized adelic measure, we will need an additional assumption on our family of maps. Define for each place $y\in Y$ and $\varphi\in S$ 
 \begin{equation}\label{eqn:C-phi-y}
  C_{\varphi}(y) = \sup_{z\in \sP^1(\bC_y)} \abs{g_{\varphi,y}(z)}_y.
 \end{equation}
This exists, as each $g_{\varphi,y}$ is a continuous function on the compact space $\sP^1_y$. Notice that, as each map $\varphi$ is defined over a single number field $K$ and has good reduction at all but finitely many places of its base field, $C_\varphi(y)$ is a nonnegative simple function with compact support in $Y$. If we take the expected value over all maps $\varphi$ in our family according to our measure $\nu_1$, we obtain a function $C_S : Y\ra \bR $ given by
\begin{equation}\label{eqn:C-constants}
 C_S(y) = \mathbb E_S C_\varphi(y) = \sum_{\varphi\in S} \nu_1(\varphi) C_\varphi(y).
\end{equation}
Notice that, as an expectation of nonnegative compactly supported simple functions, $C_S(y)$ is a nonnegative measurable extended-real valued function. We are now ready to state our condition on the family of maps:
\begin{defn}
 We say that our stochastic family of maps $(S,\nu_1)$ is \emph{$L^1$ height controlled} if the associated function $C_S : Y \ra \bR$ defined above satisfies
 \begin{equation}\label{eqn:L1-height-control}
 \int_Y C_S(y) \,d\mu(y) < \infty.
\end{equation}
\end{defn}
Notice that if a family is $L^1$ height controlled, then it is also height controlled in the sense of Healey and Hindes \cite{HealeyHindes}. 
Healey and Hindes \cite[Thm. 1.2]{HealeyHindes} defined, under this assumption, the \emph{stochastic height} $h_S$ associated to the system $(S,\nu_1)$ as
\[
 h_S(\al) = \lim_{n\ra\infty} \bE_{S^n} \frac{1}{\deg \gamma_n} h_{\gamma_n}(\al)\quad\text{for all}\quad \al\in\bP^1(\Qbar),
\]
where we recall that for a rational map $\psi$ we denote by $h_\psi$ the canonical height associated to $\psi$. In the next theorem, we will show that $h_S$ is really a height associated to a generalized adelic measure associated to $(S,\nu_1)$:
\begin{thm}\label{thm:stoch-is-gen-adelic}
 Let $S$ be a finite or countably infinite set of rational maps defined over an algebraic closure $\Qbar$, with each map being of degree at least $2$, and let $\nu_1$ be a given probability measure on $S$. If the maps in $S$ are $L^1$ height controlled, then there exists a unique generalized adelic measure $\rho$ such that the stochastic height associated to $(S,\nu_1)$ is equal to $h_\rho$.\
\end{thm}
\begin{proof}
We will construct $\rho$ by defining $\rho_y$ at almost all places $y\in Y$, and checking that the resulting measure meets the desired conditions, We start by constructing the Green's function associated to the stochastic height at $y$, continuing our construction from above. We let, for each $y\in Y$,
\begin{equation}\label{eqn:g-1}
 g_{1,y}(z) = \mathbb E_S g_{\varphi,y}(z) = \sum_{\varphi\in S} \nu_1(\varphi) g_{\varphi,y}(z),
\end{equation}
where the expectation is taken over all $\varphi$ in our probability space $(S,\nu_1)$. Let $C(y)$ be the function given by \eqref{eqn:C-constants} above. 
Notice that 
\begin{equation}
 \sup_{z\in \sP^1_y} \abs{g_{1,y}(z)} \leq \bE_S C_\varphi(y) = C_S(y).
\end{equation}

By our assumption that our family is $L^1$ height controlled, we know that the set of places $y$ where $C_S(y)=\infty$ is of $\mu$-measure $0$. Let $y\in Y$ be a place where $C_S(y)<\infty$. Then from the condition $C_S(y)<\infty$, it follows from the Weierstrass $M$-test that $g_{1,y}(z)$, is a continuous function $\sP^1(\bC_y)\ra \bR$ as well. We now recursively define, for each $n \ge 1$, the functions
\begin{equation}\label{eqn:g-n}
 g_{n+1,y}(z) = \int_{S^n}\frac{1}{\deg(\gamma_n)} g_{1,y}(\gamma_n(z))\,d\nu_n(\gamma_n) = \mathbb E_{S^n} \frac{g_{1,y}(\gamma_n(z))}{\deg(\gamma_n)}.
\end{equation}

Let us recall (see for example the discussion in \cite[\S 6.1]{FRL} or \cite[Prop. 9.54]{BakerRumelyBook}) that for any potential function $g(z)$,
\begin{equation}
 \varphi^*(\Delta g) = \Delta( g\circ \varphi).
\end{equation}
It follows that
\begin{equation}\label{eqn:g-n-plus-1-laplacian}
 \Delta g_{n+1,y} = \bE_{S^n} \frac{\gamma_n^*(\Delta g_{1,y})}{\deg(\gamma_n)},
\end{equation}
where the expectation is taken in $(S^n, \nu_n)$ and we denote the random variable by $\gamma_n$. Combining with \eqref{eqn:first-g-phi-y} and using the linearity of the pullback, we see that:
\begin{align}\label{eqn:g-n-plus-1-formula}
\begin{split}
 \Delta g_{n+1,y} &= \bE_{S^n} \bE_S \frac{\gamma_n^*(\varphi^*(\lambda_y))}{\deg(\gamma_n)\deg(\varphi)} - \bE_{S^n} \frac{\gamma_n^*(\lambda_y)}{\deg(\gamma_n)}\\
 &= \bE_{S_{n+1}} \frac{\gamma_{n+1}^*(\lambda_y)}{\deg(\gamma_{n+1})} - \bE_{S^n} \frac{\gamma_n^*(\lambda_y)}{\deg(\gamma_n)}\\
 &= \rho_{n+1,y} - \rho_{n,y},
 \end{split}
\end{align}
where we set
\begin{equation}\label{eqn:rho-n-y}
 \rho_{n,y} = \bE_{S^n} \frac{\gamma_n^*(\lambda_y)}{\deg(\gamma_n)}\quad\text{for each}\quad n\geq 1,\quad\text{and}\quad \rho_{0,y} = \lambda_y.
\end{equation}
Since the pullback of a positive Borel measure $\rho$ under a rational map $\varphi$ has mass $\varphi^*(\rho)(\sP^1_y) = \deg(\varphi) \cdot \rho(\sP^1_y)$, we have that each $\rho_{n,y}$ is a probability measure, and thus $\Delta g_{n+1,y}$ is the difference of two probability measures.

As our initial $g_{1,y}$ function is bounded, it follows that the $g_{n,y}$ functions are also bounded. Indeed, 
\[
 \abs{g_{2,y}(z)} \leq \bE_S \frac{\abs{g_{1,y}(\varphi(z))}}{\deg \varphi}\leq C_S(y) \bE_S \frac{1}{\deg \varphi} = \frac{C_S(y)}{\delta_S},
\]
where we define, following Healey and Hindes, the \emph{stochastic degree} $\delta_S$ to be the harmonic mean of the degrees of our rational maps:
\begin{equation}\label{eqn:delta-degree}
 \delta_S = \bigg( \bE_S \frac{1}{\deg \varphi}\bigg)^{-1} \geq 2,
\end{equation}
where the last inequality follows from the fact that $\deg \varphi\geq 2$ for all $\varphi\in S$ by assumption. If follows by induction using \eqref{eqn:g-n-plus-1-laplacian} that
\begin{equation}\label{eqn:geometric-bound}
 \abs{g_{n,y}(z)} \leq \frac{C_S(y)}{\delta_S^{n-1}}
\end{equation}
for all $n\geq 1$. 
Finally, we will define our stochastic telescoping series to be:
\begin{equation}\label{eqn:g-S}
 g_{S,y}(z) = \sum_{n=1}^\infty g_{n,y}(z).
\end{equation}
By the Weierstrass $M$-test and \eqref{eqn:geometric-bound}, for all $y\in Y$ for which $C_S(y)<\infty$, this series converges uniformly to a continuous function $g_{S,y} : \sP^1_y\ra\bR$. Further, if we define
\begin{equation}
 C(y) = \sup_{z\in \sP^1_y} \abs{g_{S,y}(z)},
\end{equation}
then by our telescoping construction, 
\[
 C(y) \leq 2 C_S(y)
\]
for every $y\in Y$, and so since we assumed our family $(S,\v_1)$ was $L^1$ height controlled, we have that $C_S\in L^1(Y)$, so $C\in L^1(Y)$, and $C(y)<\infty$ is true for $\mu$-almost every $y\in Y$, which is what we desired.

We now proceed to prove that, when restricted to $y\in Y(\mathbb Q,p)$, the function $g_S(y,z) = g_{S,y}(z)$ defines a measurable function
\begin{equation*}
 g_S : Y(\bQ, p) \times \sP^1(\bC_p) \ra \bR
\end{equation*}
for every rational prime $p\in M_\mathbb Q$. To see this, note that the functions $g_{\varphi}(y,z) = g_{\varphi,y}(z) : Y(\bQ,p)\times \sP^1(\bC_p)$ are measurable (in fact, continuous), as any given $\varphi$ is defined over a single number field, so they are naturally locally constant in the $y$ variable over the sets $\{Y(K,v) : v\in M_K,\ v\mid p\}$, and their potential functions are well-known to be continuous. It follows that, as a limit of measurable functions, $g_{1,y}$ is measurable in both variables, and thus likewise so are $g_{n,y}$ and $g_S$.

We now define the measures $\rho_{y}$ for each place $y\in Y$ by
\begin{equation}\label{eqn:laplacian-of-stoch-greens}
 \Delta g_{S,y}(z) = \rho_{y} - \lambda_y.
\end{equation}
This is well-defined, as the negative part of $\Delta g_{S,y}(z)$ in the Jordan decomposition is indeed $\lambda_y$, since from the first construction of $g_{\varphi,y}$ in equation \eqref{eqn:first-g-phi-y}, the negative part of the measure was $\lambda_y$. As we formed the telescoping series, the negative part remained $\lambda_y$, and as we took the expectation over $\varphi$, it remained so as well. Since the associated potential functions $g_S$ meet the criteria of Definition~\ref{defn:gen-adelic-measure}, it follows that $\rho = (\rho_y)_{y\in Y}$ is a generalized adelic measure. Further, it is clear that
\begin{equation}
 \Delta \bigg(\sum_{n=1}^N g_{n,y}\bigg) = \rho_{N,y} - \lambda_y,
\end{equation}
and since $\sum_{n=1}^N g_{n,y}(z) \ra g_{S,y}(z)$ uniformly as $N\ra\infty$, it follows (e.g., from \cite[Prop. 5.32]{BakerRumelyBook}) that $\rho_{N,y} \ra \rho_y$ in the weak sense of convergence of measures.

Lastly, we will establish that $h_{\rho} = h_S$. To see this, recall that for a rational function $\varphi$ defined over a number field $K$, if we pull back the standard adelic measure $\lambda$ \emph{once} by the map $\varphi$, we obtain an adelic measure (\emph{not} yet the canonical measure associated to $\varphi$, in general) given by 
\[
 \rho = \frac{\varphi^*(\lambda)}{\deg \varphi},
\]
and we have 
\begin{equation}\label{eqn:basic-height-pullback-formula}
 h_\rho(\al) = \frac{1}{\deg \varphi} h(\varphi(\al))
\end{equation}
for all $\al\in\bP^1(\Qbar)$. In other words, the effect of pulling back once by $\varphi$ and normalizing by the degree of $\varphi$ was to obtain the normalized standard height of $\varphi(\al)$ instead of $\al$. If we took more pullbacks under $\varphi$, we would obtain in the limit the canonical $\varphi$-height of $\al$. In our setting, it follows by taking a limit over the finite subsets of $S^n$ ordered under inclusion that the resulting height is still defined over a number field, and our formula from \eqref{eqn:basic-height-pullback-formula} applies, that the generalized adelic measures $\rho_n = (\rho_{n,y})_{y\in Y}$ satisfy
\[
 h_{\rho_n}(\al) = \bE_{S^n} \frac{1}{\deg \gamma_n} h_{\gamma_n}(\al)\quad\text{for all}\quad \al\in\bP^1(\Qbar),
\]
and so we get that $h_\rho = h_S$, as claimed. Unicity follows easily from the equidistribution result proven below in Theorem \ref{thm:stoch-equi}.
\end{proof}

\subsection{An example with places without continuous potentials}
A key fact that comes out of the proof of our equidistribution theorem for heights associated to generalized adelic measures is that it holds at the places where we have a continuous potential function $g_y : \sP^1_y \ra \bR$ with $\Delta g_y = \rho_y - \lambda_y$. Our hypothesis that our stochastic family is $L^1$ height controlled guarantees that this holds at almost every place; however, it is possible to construct families where there is a nonempty set of places, albeit still of $\mu$-measure zero, for which the potential functions are not bounded, and therefore not continuous, and therefore for which our equidistribution results will \emph{not} apply. Note that this phenomenon cannot occur with adelic measures and quasi-adelic measures, as both of these constructions are defined over a single given number field, and there are no places of $\mu$-measure zero in that context, as each place $v$ of a number field $K$ has $\mu$-measure 
\[
 \mu(Y(K,v)) = \frac{[K_v:\bQ_v]}{[K:\bQ]}>0.
\]
We now construct an example of a stochastic dynamical system where there is a nonempty, measure zero set of places for which this continuous potential function will not exist, and in particular, for which our equidistribution results will not apply.

\begin{example}\label{ex:bad-places}
Let $\tau_n$ be a sequence of Salem numbers such that $\tau_n\searrow \theta_0$ as $n\ra\infty$, where $\theta_0$ is the real root of $x^3-x-1$. That such a sequence exists is consequence of the fact that $\theta_0$ is a Pisot-Vijayaraghavan number and a result of Salem \cite[Theorem IV]{Salem}. By Northcott's theorem, as the heights of the $\tau_n$ are bounded (recall that for a Salem number $\tau>1$, $h(\tau)=(\log \tau)/[\bQ(\tau):\bQ]$), we must have $[\bQ(\tau_n):\bQ]\ra\infty$. Therefore we can choose $\tau_n$ satisfying $[\bQ(\tau_n):\bQ]\geq 3^n$ for every $n$, and we will assume that we have done so. 

Let $K_n = \bQ(\tau_1,\ldots,\tau_n)$. By our choice of $\tau_n>1$ we can assume that for each $n\in\bN$, there is a place $v_{n}$ of $K_n$ such that $\abs{\tau_{i}}_{v_{n}}>1$ for each $1\leq i\leq n$, and that $Y(K_{n+1},v_{n+1})\subset Y(K_n,v_n)$. Notice that, as an intersection of nested nonempty compact sets in a profinite space,
\[
 N := \bigcap_{n=1}^\infty Y(K_n,v_n) \neq \varnothing.
\]
We will show that at all of the places of $N$, $\rho_{S,y}$ will fail to admit a continuous potential for the stochastic dynamical system which we will define. Notice that as there is just one real place of $\bQ(\tau_n)$ where the absolute value is positive, we have 
\[
 \mu(Y(K_n,v_n)) \leq \frac{1}{[\bQ(\tau_n):\bQ]} \leq \frac{1}{3^n},
\]
so in particular, it follows that $\mu(N)=0$.

Define $d_n = [\bQ(\tau_n):\bQ]$ and let 
$
 \al_n = \tau_n^{d_n}.
$ Notice that 
\[
 h(\al_n) = \log \tau_n \searrow \log \theta_0\quad\text{as}\quad n\ra\infty.
\]
 Define for $n\in\bN$ the maps
 \[
  \varphi_n(z) = \al_n^2 z^2.
 \]
 Let $S = \{\varphi_n : n\in \bN\}$, and endow $S$ with the probability measure given by $\v_1(\varphi_n) = 1/2^n$. We will show that the family $(S,\v_1)$ is an $L^1$ height controlled family, and so there exists an adelic measure $\rho_S$ such that for $\mu$-almost every $y\in Y$, there exists a continuous function $g_y: \sP^1_y\ra\bR$ such that $\Delta g_y = \rho_{S,y} - \lambda_y$.
 
 As $\al_n$ is an algebraic unit, $\varphi_n$ has good reduction at all finite places. Thus, $\rho_{S,y}=\lambda_y$ for all $y\nmid \infty$, and $g_y$ is the identically $0$ function at all finite $y$, and hence $C(y)=0$ as well at these places. Now consider $y\in Y(\bQ,\infty)$. Notice that if we define 
 \[
  p(\gamma_k,z) = \frac{1}{\deg(\gamma_k)} \log^+\,\abs{\gamma_k(z)}_y - \log^+\,\abs{z}_y,
 \]
 then 
 \[
 \Delta p(\gamma_k, z) = \frac{\gamma_k^*(\lambda_y)}{\deg(\gamma_k)} - \lambda_y.
 \]
 It follows that, up to a constant $c\in\bR$, 
 \[
 g_y(z) = c + \lim_{k\ra\infty} \bE_{S^k}\ p(\gamma_k,z),
 \]
 where the expectation is taken over $\gamma_k\in S^k$. We say \textit{up to a constant} because, as the reader will recall, we assume that $g_y(z)$ is normalized at every place so that $g_y(z)\ra 0$ as $z\ra\infty$, but as we will observe, that requires here adjusting our choice of $p$ function. To see this, note that if $\gamma_k = (\varphi_{n_1}, \ldots, \varphi_{n_k})$, then 
 \[
 p(\gamma_k,z) = \log^+\, \abs{\al_{n_k}^{1/2^{k-1}} \cdots \al_{n_2}^{1/2} \al_{n_1} z}_y - \log^+\,\abs{z}_y,
 \]
 so in order to assume that $p(\gamma_k,z)\ra 0$ as $z\ra\infty$, we define
 \[
 p_0(\gamma_k,z) =  p(\gamma_k,z) - \log\,\abs{\al_{n_k}^{1/2^{k-1}} \cdots \al_{n_2}^{1/2} \al_{n_1}}_y,
 \]
 and then we see that $g_y(z) = \lim_{k\ra\infty} \bE_{S^k}\ p_0(\gamma_k,z)$ and $g_y$ is normalized as we desired. It easy to check that for $c>0$, the function $\abs{\log^+\,\abs{cz}_y-\log^+\,\abs{z}_y - \log c}$ has maximum $\abs{\log c}$, and so it follows that 
\[
 \sup_{z\in\sP^1_y} \abs{p_0(\gamma_k,z)} = \abs{p_0(\gamma_k,0)} = \big| \log\, \abs{\al_{n_k}^{1/2^{k-1}} \cdots \al_{n_2}^{1/2} \al_{n_1}}_y\big|.
\]
 Now as each component of $\gamma_k$ is chosen independently in $S$, it follows that 
 \[
 \bE_{S^k} \abs{p_0(\gamma_k,0)} = \frac{2^k-1}{2^k} \sum_{n=1}^\infty \v_1(\varphi_n) \log\, \abs{\al_n}_y = \frac{2^k-1}{2^k} \sum_{n=1}^\infty \frac{1}{2^n} \big|\log\, \abs{\al_n}_y\big|.
 \]
 So in particular, we get that
 \begin{equation}
  C(y) = \sup_{z\in\sP^1_y} \abs{g_y(z)} = \abs{g_y(0)} \leq  \sum_{n=1}^\infty \frac{1}{2^n} \big|\log\, \abs{\al_n}_y\big|.
 \end{equation}
 Notice that for all $y\in N$, we have $\abs{p_0(\gamma_k,0)} = -p_0(\gamma_k,0)$ for every $\gamma_k$, as $\abs{\al_n}_y>1$ for all $n\geq 1$ and $y\in N$, so in fact
 \[
 -g_y(0) = \sum_{n=1}^\infty \frac{1}{2^n} \log\, \abs{\al_n}_y\geq 
 \sum_{n=1}^\infty \frac{1}{2^n} \log\, d_n \log \tau_n \geq \sum_{n=1}^\infty \frac{1}{2^n} \log\, 3^n \log \theta_0 = \infty,
 \]
 so our potential functions cannot be bounded for the places $y\in N$. 
 
 Now it remains to show that $C(y)$ is still an integrable function on $Y(\bQ,\infty)$ in order to prove that our family is still integrable. To see this, notice that by the product formula, 
 \[
 \int_{Y(\bQ,\infty)} \log\, \abs{\al_n}_y\,d\mu(y) = 0,
 \]
 and thus, as our $\al_n$ is an algebraic unit, 
 \[
  \int_{Y(\bQ,\infty)} \big|\log\, \abs{\al_n}_y\big| \,d\mu(y) = 2h(\al_n)
 \]
 (cf. \cite[Equation 1.3]{AV}). It follows that 
 \begin{align*}
 \int_Y C(y)\,d\mu(y) &\leq  \sum_{n=1}^\infty \frac{1}{2^n} \int_{Y(\bQ,\infty)} \big|\log\, \abs{\al_n}_y\big|\,\,d\mu(y)\\
 &= \sum_{n=1}^\infty \frac{2h(\al_n)}{2^n} \leq 2\sup_{n}\log \tau_n,
 \end{align*}
 but this quantity is bounded as $\tau_n \searrow \theta_0$. It follows that $(S,\v_1)$ is $L^1$ height controlled as claimed. 
 
 The reader may find it interesting to note as well that the Borel-Cantelli lemma implies that for $\mu$-almost all $y\in Y(\bQ,\infty)$, $\abs{\al_n}_y=1$ for all but finitely $n$. We note that our set $N$ lies in the exceptional set where infinitely many $\abs{\al_n}_y>1$.
\end{example}

\subsection{Pullback formula for stochastic heights}
The stochastic generalized adelic measure shares some, but not all, of the properties of the usual dynamical measure associated to iteration of a single rational map. For example, if $\mu_{\varphi,y}$ denotes the canonical measure of $\varphi$ at a place $y$, then it is well-known (cf. \cite[Theorem 10.2]{BakerRumelyBook}) that
\[
 \mu_{\varphi,y} = \frac{\varphi^*(\mu_{\varphi,y})}{\deg \varphi} \quad\text{and}\quad \varphi_*(\mu_{\varphi,y}) = \mu_{\varphi,y}.
\]
We will prove a similar statement for the expected pullback of the stochastic canonical measure. However, the analogous stochastic pushforward formula does \emph{not} hold in general for stochastic heights, and this is tied to the fact that the stochastic Fatou and Julia sets are not fully invariant under the maps $\varphi \in S$, which will be proved in the forthcoming paper \cite{DFT_II}.
\begin{thm}[Pullback formula for stochastic heights]\label{thm:pullback-formula}
 Let $(S,\nu_1)$ be an $L^1$ height controlled stochastic family of rational maps, each of degree at least $2$, and let $\rho$ denote the canonical generalized adelic measure associated to $S$. Then for $\mu$-almost every place $y\in Y$,
 \[
  \rho_y = \bE_S \frac{\varphi^*(\rho_y)}{\deg \varphi}, 
 \]
 where the expectation is taken over all maps $\varphi\in S$ with respect to the measure $\nu_1$.
\end{thm}
We note that, as usual, in the case that $\rho$ is defined over a single number field $K$, the $\mu$-almost every place condition becomes every place of $K$, as every place $v$ of $K$ gives rise to a set $Y(K,v)$ of positive $\mu$ measure.
\begin{proof}
 Recall from above that we defined in \eqref{eqn:g-S} the Green's functions $g_{S,y}$, whose Laplacians satisfy $\Delta g_{S,y}(z) = \rho_{y} - \lambda_y$ at each place, provided that $C(y)<\infty$ at that place. As $C\in L^1(Y)$ by assumption that $S$ is $L^1$ height controlled, this condition holds for $\mu$-almost every place $y$. We will use the same notation as in the proof of Theorem \ref{thm:stoch-is-gen-adelic}, that is, the partial sum of $g_{S,y}(z)$ is denoted
 \[
  g_{S,y,N}(z) = \sum_{n=1}^N g_{n,y}(z),
 \]
 where for each $n\geq 1$, $g_{n,y}$ has Laplacian 
 \[
  \Delta g_{n,y} = \rho_{n,y}-\rho_{n-1,y},
 \]
 where $\rho_{0,y} = \lambda_y$ and $\rho_{n,y}$ is defined in \eqref{eqn:rho-n-y} as the expected pullback measure of $\lambda_y$ under the maps $\gamma\in S^N$ weighted by $1/\deg(\gamma)$.  It follows that the partial sum forms a telescoping series and has Laplacian
 \[
  \Delta g_{S,y,N} = \rho_{N,y} - \lambda_y.
 \]
 Notice that, as occurred above in the computation of  \eqref{eqn:g-n-plus-1-formula}, 
 \[
  \bE_S \frac{\varphi^*(\rho_n)}{\deg \varphi} = \rho_{n+1,y},
 \]
 so it follows that
 \[
  \bE_S \frac{\varphi^*(\Delta g_{S,y,N})}{\deg \varphi} = \rho_{N+1,y} - \rho_{1,y}.
 \]
 Taking the limit as $N\ra\infty$, we see that 
 \begin{equation}\label{eqn:pullback-of-g-S}
  \bE_S \frac{\varphi^*(\Delta g_{S,y})}{\deg \varphi} = \rho_{y} - \rho_{1,y}.
 \end{equation}
 But as $
  \Delta g_{S,y} = \rho_y-\lambda_y,
 $
 we have
 \begin{align}\label{eqn:pullback-of-g-S-direct}
  \begin{split}
  \bE_S \frac{\varphi^*(\Delta g_{S,y})}{\deg \varphi} 
  &=  \bE_S\frac{\varphi^*(\rho_{y})}{\deg \varphi} - \bE_S \frac{\varphi^*(\lambda_y)}{\deg \varphi}\\
  &= \bE_S\frac{\varphi^*(\rho_{y})}{\deg \varphi} - \rho_{1,y}. 
  \end{split}
 \end{align}
 Combining equations \eqref{eqn:pullback-of-g-S} and \eqref{eqn:pullback-of-g-S-direct}, we see that
 \begin{equation}
  \bE_S \frac{\varphi^*(\rho_y)}{\deg \varphi} = \rho_y,
 \end{equation}
 which is what we wanted to show. 
\end{proof}

\subsection{Stochastic dynamical equidistribution}
We are now ready to state our main equidistribution theorem. We start by defining the random backwards orbit measure of a point:
\begin{defn}
 Let $\al\in\Qbar$. We define the \emph{random backwards orbit measure} of $\al$ under the stochastic dynamical system $(S,\nu_1)$ recursively by 
 $
  \Delta_{0,\al} = \delta_\al, 
 $
 and 
 \[
  \Delta_{n+1,\al} = \bE_S \bigg( \frac{\varphi^*(\Delta_n)}{\deg(\varphi)}\bigg) \quad\text{for}\quad n\geq 0.
 \]
 Here, $\delta_z$ denotes the Dirac point mass at $z$, and $\varphi^*(\Delta)$ denotes the pullback of the measure $\Delta$ under $\varphi$. Note that for a point mass $\delta_z$ we have
 \[
  \varphi^*(\delta_z) = \sum_{w\in \varphi^{-1}(z)} \delta_w,
 \]
 where the $w$ in the sum are counted with multiplicity, so that we must divide by the degree of $\varphi$ to obtain a new probability measure.
\end{defn}
Notice that, so long as the measure $\v_1$ is strictly positive, the support of $\Delta_{n.\alpha}$ is
 \[
  \supp(\Delta_{n,\al}) = \bigcup_{\gamma_n\in S^n} \gamma_n^{-1}(\al),
 \]
and the weights come from the probability of the individual likelihood of each $\gamma_n$ in $S^n$ under the product measure, along with the multiplicities as preimages under $\gamma_n$.

\begin{thm}[Stochastic dynamical equidistribution]\label{thm:stoch-equi}
 Let $S$ be a countable set of rational maps defined over an algebraic closure $\Qbar$, with each map being of degree at least $2$, and let $\nu_1$ be a probability measure on $S$ with respect to which the maps in $S$ are $L^1$ height controlled. Suppose $\al\in \bP^1(\Qbar)\setminus E_S$ where $E_S$ denotes the exceptional set of the stochastic dynamical system. Then for $\mu$-almost every place $y$ of $\Qbar$, the backwards orbit measures $\Delta_{n,\al}$ under $(S,\v_1)$ converge weakly to the stochastic dynamical measure $\rho_y$ as $n\ra\infty$.
\end{thm}
The substance of the theorem lies in proving that the $\Delta_{n,\al}$ measures are well-distributed and have $h_S(\Delta_{n,\al})\ra 0$. The equidistribution result then follows from applying Theorem \ref{thm:gen-adelic-equi} above, the equidistribution theorem for generalized adelic measures.

Before beginning our proof of stochastic dynamical equidistribution, we need to prove a few results. We start with a result which generalizes the main conclusion of Theorem 1.2(b) from \cite{HealeyHindes}.
\begin{prop}\label{prop:stochastic-degree-scaling}
 Let $(S,\v_1)$ be an $L^1$ height controlled stochastic family of rational maps, each of degree at least $2$. Then for any discrete probability measure $\Delta$ on $\bP^1(\Qbar)$ of finite height, 
 \[
  \bE_S \frac{h_S(\varphi_*(\Delta))}{\deg \varphi} = h_S(\Delta),
 \]
 where $\varphi_*(\Delta)$ denotes the pushforward measure. 
\end{prop}
 \noindent We remind the reader that the pushforward measure of a discrete measure 
 \[
  \Delta = \sum_i t_i \delta_{\al_i}
 \]
 where  $t_i\geq 0$, $\sum t_i = 1$, and $\al_i\in \bP^1(\Qbar)$, is given by
 \[
 \varphi_*(\Delta) = \sum_i t_i \delta_{\varphi(\al_i)}.
 \]
\begin{proof}
 By Proposition \ref{prop:height-average}, it suffices to prove the result for $\Delta = \delta_\al$ for $\al\in\Qbar$. (We note in passing that $h(\delta_\al)$ is indeed equal to the usual height of $\al$, see Proposition \ref{prop:galois-invariance-of-height} above.) Recall that
 \[
  h_S(\al) = \lim_{n\ra\infty} \int_{S^n} \frac{h(\gamma_n(\al))}{\deg \gamma_n}\,d\nu_n(\gamma_n).
 \]
 For $n\geq 2$, each $\gamma_n\in S^n$ can be written as $\gamma_n = \gamma_{n-1} \circ \varphi$ where $\gamma_{n-1}\in S^{n-1}$ and $\varphi \in S$. By the absolute convergence of the series involved,
 \begin{align*}
  h_S(\al) &= \lim_{n\ra\infty} \int_{S^{n-1}} \int_S \frac{h(\gamma_{n-1}\circ \varphi (\al))}{\deg (\gamma_{n-1}) \deg(\varphi)}\,d\v_1(\varphi) \,d\nu_{n-1}(\gamma_{n-1})\\
  &= \int_S \frac{1}{\deg(\varphi)} \lim_{n\ra\infty} \int_{S^{n-1}} \frac{h(\gamma_{n-1}\circ \varphi (\al))}{\deg (\gamma_{n-1})}\,d\nu_{n-1}(\gamma_{n-1})\,d\v_1(\varphi)\\
  &= \int_S \frac{1}{\deg(\varphi)} h_S(\varphi(\al)) \,d\v_1(\varphi) \\
  &= \bE_S \frac{h_S(\varphi_*(\al))}{\deg \varphi}.\qedhere
 \end{align*}
\end{proof}
Before we move on, we state for its independent interest a restatement of Proposition \ref{prop:stochastic-degree-scaling} in the case where $\Delta = \delta_\al$:
\begin{cor}
If $(S,\v_1)$ is an $L^1$ height controlled family of dynamical rational maps, each of degree at least $2$, and $\al\in \bP^1(\Qbar)$, then the stochastic height $h_S(\al)$ satisfies
 \begin{equation}
  h_S(\al) = \bE_S \frac{h_S(\varphi(\al))}{\deg \varphi}.
 \end{equation}
\end{cor}
 \noindent We note that, in the case where the family $S$ meets the conditions of the \cite{HealeyHindes} paper, this corollary is essentially equivalent to the $k=1$ case of their Theorem 1.2(b). We state it primarily because it is conceptually important and aesthetically pleasing as the stochastic analogue of the usual scaling property satisfied by the dynamical height, namely, that 
 \[
  h_\varphi(\al) = \frac{h_\varphi(\varphi(\al))}{\deg\varphi}.
 \]
 
 \begin{lemma}\label{lemma:backwards-orbit-has-finite-height}
 Let $(S,\v_1)$ be an $L^1$ height controlled stochastic family of rational maps, each of degree at least $2$. Let $\Delta$ be a discrete measure on $\bP^1(\Qbar)$ with finite height. Then
 \[
  \Gamma = \bE_S \left(\frac{\varphi^*(\Delta)}{\deg(\varphi)}\right)
 \]
 is a discrete measure on $\bP^1(\Qbar)$ with finite height as well, and there exists a constant $C\geq 0$ (independent of $\Delta$) such that 
 \[
  h(\Gamma) \leq \frac12 h(\Delta) + C.
 \]
 \end{lemma}
 \begin{proof}
 We begin by recalling the notation of Theorem \ref{thm:gen-h-is-weil-nonneg} above. For each $\varphi\in S$ we have an generalized adelic measure $\rho_\varphi$ defined over the field of definition of $\varphi$. In particular, for each $y\in Y$ there exists a continuous function $g_y : \sP^1_y \ra \bR$ such that $\Delta g_{\varphi,y} = \rho_{\varphi,y} - \lambda_y$, and further, the function
 \[
  C_\varphi(y) = \sup_{z\in \sP^1_y} \abs{g_{\varphi,y}(z)}
 \]
 is compactly supported and  constant on the sets $Y(K,v)$, where $K$ is a number field over which $\varphi$ is defined and $v$ ranges over the places of $K$. In particular, $C_\varphi\in L^1(Y)$. In fact, we proved in equation  \eqref{eqn:explicit-height-bound} that for all discrete measures $\Delta$ on $\Qbar$, we have
 \begin{equation*} 
  \abs{ h_{\varphi}(\Delta) - h(\Delta)} \leq 3 \int_Y C_\varphi(y)\,d\mu(y).
 \end{equation*}
 Further, the condition that $S$ is $L^1$ height bounded states that the function
 \[
 C(y) = \bE_S C_\varphi(y) 
 \]
is also in $L^1(Y)$.
Now, 
\[
 h_{\varphi}(\varphi^*(\Delta)) = h_{\varphi}(\Delta).
\]
 (Recall that while $h_\varphi(\beta) = (1/\deg(\varphi)) h_\varphi(\al)$ for any $\beta\in \varphi^{-1}(\al)$, the pullback measure has total mass $\deg(\varphi)$ as it weights each point in the pullback equally with multiplicity.) It follows from Proposition~\ref{prop:height-average} that

 \begin{align*}
  h(\Gamma) =  \bE_S \left(\frac{h(\varphi^*(\Delta))}{\deg(\varphi)}\right)
  &\leq \bE_S\frac{1}{\deg(\varphi)} \left(h_\varphi(\varphi^*(\Delta)) + 3\int_Y C_\varphi \,d\mu\right)\\
  &= \bE_S\frac{1}{\deg(\varphi)} \left(h_\varphi(\Delta) + 3\int_Y C_\varphi \,d\mu\right)\\
  &\leq \bE_S\frac{1}{\deg(\varphi)} \left(h(\Delta) + 6\int_Y C_\varphi \,d\mu\right)\\
  &\leq \frac12 h(\Delta) + 3 \int_Y C(y)\,d\mu(y),
 \end{align*}
 where we have used the fact that $\deg(\varphi)\geq 2$ for every $\varphi\in S$, which completes the proof of the claim with constant $C = 3 \int_Y C(y)\,d\mu(y)$. 
 \end{proof}
 
\begin{prop}\label{prop:backwards-orbit-has-finite-height}
  Let $(S,\v_1)$ be an $L^1$ height controlled stochastic family of rational maps, each of degree at least $2$. For any $\al\in\bP^1(\Qbar)$, the backwards orbit measures $\Delta_{n,\al}$ have finite height in the sense of Definition \ref{defn:Delta-height-bounded} above.
\end{prop}
\begin{proof}
 It suffices to prove finite height with respect to the classical height, as by Corollary \ref{cor:finite-height-is-same-for-all-heights}, a discrete $\bP^1(\Qbar)$ measure has finite height with respect to one generalized adelic measure $\rho$ if and only if it has finite height for all such heights. Note that $\Delta_{0,\al} = \delta_\al$ has finite height $h(\Delta_{0,\al}) = h(\al)$. Suppose that $\Delta_{n,\al}$ has finite height. Then by applying Lemma \ref{lemma:backwards-orbit-has-finite-height}, we see that
 \[
 h(\Delta_{n+1,\al}) \leq \frac12 h(\Delta_{n}(\al)) + C
 \]
 for $C$ a constant which depends only on the family $(S,\v_1)$, but not on $\al$. It follows that each $\Delta_{n+1,\al}$ has finite height, as claimed.
\end{proof}
Indeed, it follows from the proof above that the limit supremum of the standard heights of a backwards orbit is bounded, but we will prove in the course of our equidistribution theorem that in fact $h_S(\Delta_{n,\al})\ra 0$ as $n\ra\infty$.

\begin{prop}\label{prop:backwards-orbits-are-well-distr}
 Let $(S,\v_1)$ be an $L^1$ height controlled stochastic family of rational maps, each of degree at least $2$. For any $\al\in \bP^1(\Qbar)\setminus E_S$ not in the exceptional set of $S$, the backwards orbit measures $\Delta_{n,\al}$ are well-distributed.
\end{prop}
\noindent We refer the reader to Definition~\ref{defn:exceptional-set} above for the definition of the 
exceptional set of $S$. 

In order to prove Proposition~\ref{prop:backwards-orbits-are-well-distr}, we require a few lemmas. For the remainder of this section, suppose that $S$ is as in Proposition \ref{prop:backwards-orbits-are-well-distr} and for convenience let $\Delta_n=\Delta_{n,\al}$. For a rational function $\psi$, we define $e_z(\psi)$ to be ramification index of $\psi$ at $z$.

\begin{lemma}
Let $S,\al$ and $\Delta_n=\Delta_{n,\al}$ be as in Proposition~\ref{prop:backwards-orbits-are-well-distr} above. Assume $n \ge 0$ is fixed. Then for each $0 \le k \le n$, we have
	\[
		\Delta_n(z) = \sum_{\gamma_k \in S^k} \nu_k(\gamma_k) \frac{e_z(\gamma_k)}{\deg \gamma_k} \Delta_{n-k}(\gamma_k(z)).
	\]
\end{lemma}
\begin{proof}
The proof is by induction on $n$. Note that the case $n = 0$ is trivial, as we identify $S^0$ as consisting of just the identity map and with $\nu_0$ as the trivial probability measure. Moreover, for any $n \ge 0$ the case $k = 0$ is trivial for the same reason, so we will henceforth assume $k \ge 1$. Now let $n \ge 1$. We have
	\begin{align*}
	\Delta_n(z)
		&= \bE_S\left(\frac{\varphi^*(\Delta_{n-1})(z)}{\deg \varphi}\right)\\
		&= \sum_{\varphi \in S} \nu_1(\varphi) \frac{\varphi^*(\Delta_{n-1})(z)}{\deg \varphi}\\
		&= \sum_{\varphi \in S} \frac{\nu_1(\varphi)}{\deg\varphi} \sum_{x \in \P^1(\Qbar)} \Delta_{n-1}(x)\sum_{w \in \varphi^{-1}(x)} e_w(\varphi)\delta_w(z)\\
		&= \sum_{\varphi \in S} \nu_1(\varphi)\frac{e_z(\varphi)}{\deg\varphi} \Delta_{n-1}(\varphi(z)).
	\end{align*}
By induction, for all $w \in \P^1(\Qbar)$ we have for $1 \le k \le n$ that
	\[
		\Delta_{n-1}(w) = \sum_{\gamma_{k-1} \in S^{k-1}}\nu_{k-1}(\gamma_{k-1})\frac{e_w(\gamma_{k-1})}{\deg \gamma_{k-1}} \Delta_{n-k}(\gamma_{k-1}(w));
	\]
 applying this to $w = \varphi(z)$ with $\varphi \in S$ yields
	\begin{align*}
		\Delta_n(z)
			&= \sum_{\varphi \in S} \nu_1(\varphi)\frac{e_z(\varphi)}{\deg\varphi}\sum_{\gamma_{k-1} \in S^{k-1}}\nu_{k-1}(\gamma_{k-1})\frac{e_{\varphi(z)}(\gamma_{k-1})}{\deg \gamma_{k-1}} \Delta_{n-k}(\gamma_{k-1}(\varphi(z)))\\
			&= \sum_{\varphi \in S}\sum_{\gamma_{k-1} \in S^{k-1}} \nu_1(\varphi)\nu_{k-1}(\gamma_{k-1})\frac{e_z(\varphi)}{\deg\varphi}\frac{e_{\varphi(z)}(\gamma_{k-1})}{\deg \gamma_{k-1}} \Delta_{n-k}(\gamma_{k-1}(\varphi(z))),
	\end{align*}
and by writing $\gamma_k = \gamma_{k-1} \circ \varphi \in S^k$ we get
	\[
		\Delta_n(z) = \sum_{\gamma_k \in S^k} \nu_k(\gamma_k)\frac{e_z(\gamma_k)}{\deg \gamma_k} \Delta_{n-k}(\gamma_k(z)).\qedhere
	\]
\end{proof}

\begin{lemma}\label{lem:level3ramification}
Let $S,\al$ and $\Delta_n=\Delta_{n,\al}$ be as in Proposition~\ref{prop:backwards-orbits-are-well-distr} above. Suppose $e_z(\gamma_3) = \deg \gamma_3$ for all $\gamma_3 \in S^3$. Then $z \in E_S$.
\end{lemma}

\begin{proof}
Define the set
	\[
		\cR_S = \{x \in \P^1(\Qbar) : e_x(\varphi) = \deg \varphi \text{ for all } \varphi \in S\}.
	\]
By the Riemann-Hurwitz theorem, we have $|\cR_S| \le 2$.

Let $\varphi_1,\varphi_2,\varphi_3 \in S$ be arbitrary. Taking $\gamma_3 = \varphi_3 \circ \varphi_2 \circ \varphi_1$, the equality $e_z(\gamma_3) = \deg \gamma_3$ implies that 
	\[
		e_z(\varphi_1) = \deg\varphi_1,\quad e_{\varphi_1(z)}(\varphi_2) = \deg\varphi_2,\quad\text{and}\quad e_{\varphi_2(\varphi_1(z))}(\varphi_3) = \deg\varphi_3.
	\]
Since $\varphi_1,\varphi_2,\varphi_3$ were arbitrary, it follows that
	\begin{enumerate}
	\item $z \in \cR_S$;
	\item $\varphi(z) \in \cR_S$ for all $\varphi \in S$; and
	\item $\varphi_2(\varphi_1(z)) \in \cR_S$ for all $\varphi_1,\varphi_2 \in S$.
	\end{enumerate}

Suppose that $\varphi(z) = z$ for all $\varphi \in S$. Then, since $\varphi$ is totally ramified at $z$, we have $\varphi^{-1}(z) = \{z\}$ for all $\varphi \in S$, and therefore $z \in E_S$.

Now suppose there exists $\varphi_1 \in S$ such that $w := \varphi_1(z) \ne z$. By (b), we have $\varphi(z) \in \cR_S$ for all $\varphi \in S$; in particular, we have $w \in \cR_S$, so $\cR_S = \{z,w\}$. By (c), we also have $\varphi(w) \in \cR_S$ for all $\varphi \in S$. Thus, for all $\varphi \in S$, we have $\varphi(z) \in \{z,w\}$ and $\varphi(w) \in \{z,w\}$. Now fix $\varphi \in S$; since $\varphi$ is totally ramified at both $z$ and $w$, we have $\varphi(z) \ne \varphi(w)$, hence $\varphi(\cR_S) = \cR_S$, and therefore---once again using the fact that $\varphi$ is totally ramified at $z$ and $w$---we have $\varphi^{-1}(\cR_S) = \cR_S$. Therefore, $z$ is exceptional.
\end{proof}

\begin{rmk}
The conclusion of Lemma~\ref{lem:level3ramification} need not hold if we only assume that $e_z(\gamma_2) = \deg \gamma_2$ for all $\gamma_2 \in S^2$. For example, let $\phi_1(z) = \frac{1}{z^2}$ and $\phi_2(z) = z^2 + 1$, and take $S = \{\phi_1,\phi_2\}$. For all four second iterates $\gamma_2 \in S^2$, we have $e_\infty(\gamma_2) = 4 = \deg \gamma_2$. However, $\infty$ is not exceptional; see Remark~\ref{rem:exceptional}.
\end{rmk}

\begin{lemma}\label{lem:bounded_away_from_1}
Let $S,\al$ and $\Delta_n=\Delta_{n,\al}$ be as in Proposition~\ref{prop:backwards-orbits-are-well-distr} above. For $z \in \P^1(\Qbar)$, define
	\[
		\sigma(z) = \sum_{\gamma_3 \in S^3} \nu_3(\gamma_3)\frac{e_z(\gamma_3)}{\deg \gamma_3}.
	\]
There exists $M < 1$ such that $\sigma(z) < M$ for all $z \notin E_S$.
\end{lemma}

\begin{proof}

Fix $\gamma_3 \in S^3$, and define
    \[
        \cA := \{z \in \P^1(\Qbar) : e_z(\gamma_3) = \deg \gamma_3\}.
    \]
By the Riemann-Hurwitz theorem, the set $\cA$ has at most two elements, so in particular it is finite.

If $z \notin \cA$, we have $e_z(\gamma_3) < \deg \gamma_3$, hence
    \[
        \sigma(z) \le M_1 := 1 - \frac{\nu_3(\gamma_3)}{\deg \gamma_3} < 1.
    \]
If $z \in \cA\setminus E_S$, there exists (by Lemma~\ref{lem:level3ramification}) an element $\gamma_3' \in S^3$ such that $e_z(\gamma_3') < \deg \gamma_3'$, and arguing as above yields $\sigma(z) < 1$. Since $\cA$ is finite, the quantity
    \[
        M_2 := \max_{z \in \cA \setminus E_S} \sigma(z)
    \]
exists and is strictly less than $1$.
The lemma follows by taking $M := \max\{M_1,M_2\}$.
\end{proof}

We are now ready to prove Proposition~\ref{prop:backwards-orbits-are-well-distr}.

\begin{proof}[Proof of Proposition~\ref{prop:backwards-orbits-are-well-distr}]
It suffices to show that
	\[
		\lim_{n\to\infty} \sup_{z \in \P^1(\Qbar)} \Delta_n(z) = 0,
	\]
since we have
	\[
		\sum_{z \in \P^1(\Qbar)} \Delta_n(z)^2
			\le \bigg(\sup_{z \in \P^1(\Qbar)} \Delta_n(z)\bigg) \cdot \sum_{z \in \P^1(\Qbar)} \Delta_n(z)\\
			= \sup_{z \in \P^1(\Qbar)} \Delta_n(z).
	\]
First, we observe that for all $z \in \P^1(\Qbar)$ we have
	\begin{align*}
	\Delta_n(z)
		&= \sum_{\varphi \in S} \nu_1(\varphi)\frac{e_z(\varphi)}{\deg\varphi} \Delta_{n-1}(\varphi(z))\\
		&\le \bigg(\sup_{w \in \P^1(\Qbar)} \Delta_{n-1}(w)\bigg)\sum_{\varphi \in S} \nu_1(\varphi)\\
		&= \sup_{w \in \P^1(\Qbar)} \Delta_{n-1}(w),
	\end{align*}
so the sequence $\bigg(\displaystyle\sup_{z \in \P^1(\Qbar)} \Delta_n(z)\bigg)_{n\in\bN}$ is non-increasing.

Finally, we claim that the subsequence $\bigg(\displaystyle\sup_{z \in \P^1(\Qbar)} \Delta_{3k}(z)\bigg)_{k\in\bN}$ decreases at least geometrically, from which our result follows. Observe that since $\alpha \notin E_S$, $\Delta_n$ is not supported on $E_S$ for any $n$; that is, $\Delta_n(w) = 0$ for all $w \in E_S$.

Now fix $k \ge 1$ and $z \in \P^1(\Qbar) \setminus E_S$. Let $M$ be as in Lemma~\ref{lem:bounded_away_from_1}. Then we have
	\begin{align*}
	\Delta_{3k}(z)
		&= \sum_{\gamma_3 \in S^3}\v_3(\gamma_3)\frac{e_z(\gamma_3)}{\deg \gamma_3} \Delta_{3(k-1)}(\gamma_3(z))\\
		&\le \bigg(\sup_{w \in \P^1(\Qbar)} \Delta_{3(k-1)}(w)\bigg)\sum_{\gamma_3 \in S^3}\v_3(\gamma_3)\frac{e_z(\gamma_3)}{\deg \gamma_3}\\
		&= \bigg(\sup_{w \in \P^1(\Qbar)} \Delta_{3(k-1)}(w)\bigg) \cdot \sigma(z)\\
		&\le M\bigg(\sup_{w \in \P^1(\Qbar)} \Delta_{3(k-1)}(w)\bigg).
	\end{align*}
Since $z$ was arbitrary, we have
	\[
		\sup_{z \in \P^1(\Qbar)} \Delta_{3k}(z) \le
		M\bigg(\sup_{z \in \P^1(\Qbar)} \Delta_{3(k-1)}(z)\bigg),
	\]
and since $M < 1$, we have the claimed geometric decay.
\end{proof}

We are now ready to prove our stochastic equidistribution theorem:
\begin{proof}[Proof of Theorem \ref{thm:stoch-equi}]
 The theorem will follow primarily from demonstrating that
 \[
  h_S(\Delta_{n,\al}) \ra 0 \quad\text{as}\quad n\ra\infty,
 \]
 so we will start by showing this. Notice that, by construction of our random backwards orbit measures $\Delta_{n,\al}$, we have
 \[
  \bE_S \varphi_* (\Delta_{n+1,\al}) = \Delta_{n,\al} \quad\text{for all}\quad n\geq 0.
 \]
 Applying Proposition \ref{prop:stochastic-degree-scaling}, we see that 
 \[
  h_S(\Delta_{n+1,\al}) = \bE_S \frac{h_S(\varphi_*(\Delta_{n+1,\al}))}{\deg \varphi} \leq \frac12 \bE_S h_S(\varphi_*(\Delta_{n+1,\al})),
 \]
 as $\deg\varphi \geq 2$ for every $\varphi\in S$. Now, applying Lemma \ref{prop:height-average}, we can say that
 \begin{equation}
  \bE_S h_S(\varphi_*(\Delta_{n+1,\al})) = h_S\left(\bE_S \varphi_*(\Delta_{n+1,\al})\right) = h_S(\Delta_{n,\al}).
 \end{equation}
 It follows that 
 \[
  h_S(\Delta_{n+1,\al}) \leq \frac12 h_S(\Delta_{n,\al}),
 \]
 so by induction,
 \[
  h_S(\Delta_{n,\al}) \leq \frac{1}{2^n} h_S(\al)\quad\text{for all}\quad n\geq 0.
 \]
 It follows that $h_S(\Delta_{n,\al}) \ra 0$ as claimed. The $h_S$ height is given by a generalized adelic measure by \ref{thm:stoch-is-gen-adelic}, and by Proposition \ref{prop:backwards-orbits-are-well-distr}, the $\Delta_{n,\al}$ measures are well-distributed. Theorem \ref{thm:gen-adelic-equi} therefore applies and gives the desired result.
\end{proof}

\subsection{Stochastic Julia and Fatou sets}
Let $(S,\nu_1)$ be an $L^1$ height controlled stochastic family of rational maps, each of degree at least $2$, and suppose that $\nu_1$ is strictly positive. Let $\rho$ denote the canonical generalized adelic measure such that $h_S = h_\rho$, as constructed above. For any place $y\in Y$ for which $C(y)<\infty$, we define the \emph{stochastic Julia set} at the place $y$ to be
\[
 \cJ_{S,y} = \supp(\rho_y),
\]
and we define the \emph{stochastic Fatou set} to be 
\[
 \cF_{S,y} = \sP^1_y \setminus \cJ_{S,y}.
\]
 Many of the usual results one expects for Julia sets and Fatou sets are not always true for stochastic systems. We conclude with the following simple example which illustrates some of the key differences. 
\begin{example}
 Let $S = \{\varphi, \psi\}$ where $\varphi(z) = z^2$, $\psi(z) = 2z^2$ with $\nu_1(\varphi) = \nu_1(\psi) = 1/2$. Notice that $S$ is defined over $\bQ$, and since $S$ contains only finitely many maps, there are only finitely many bad places (namely, $p=2,\infty$), and so the associated generalized adelic measure is also an adelic measure over $\bQ$ in the original sense of Favre and Rivera-Letelier. 

 We claim that, at the archimedean place, the Julia set is the annulus
 \[
  \cJ_{S,\infty} = \{ z\in \bC : 1/2\leq \abs{z}\leq 1\}
 \]
 and $\rho_{S,\infty}$ is the radial measure which satisfies
 \[
  \int_{D(0,r)} d\rho_{S,\infty} = 
  \begin{cases}
   0             & \text{if }r \leq 1/2,\\
   1 + \log_2(r) & \text{if }1/2\leq r \leq 1,\\
   1             & \text{if }r \geq 1,
  \end{cases}
 \]
 or equivalently,
 \[
  d\rho_{S,\infty}(re^{i\theta}) = \begin{cases}
   \displaystyle \frac{dr}{r\log 2}\frac{d\theta}{2\pi} &\text{if }1/2\leq r\leq 1,\\
   0 &\text{otherwise}.
  \end{cases}
 \]
 To see this, note that 
 \[
  g_{1,\infty}(z) = \frac12\left(\frac12 \log^+\,\abs{2z^2} + \frac12 \log^+ \abs{z^2}\right) - \log^+\,\abs{z}.
 \]
 If we denote the $n$th approximation to the local height by
 \[
  \hat\lambda_{n,\infty}(z) = \bE_{S^n} \frac{1}{\deg \gamma_n} \log^+\,\abs{\gamma_n(z)},
 \]
 and let $\hat\lambda_{0,\infty}(z) = \log^+\,\abs{z}$ by convention, then our $n$th term in the telescoping series that determines the Green's function is 
 \[
  g_{n,\infty}(z) = \hat\lambda_{n,\infty}(z) - \hat\lambda_{n-1,\infty}(z)\quad\text{for all}\quad n\geq 1.
 \]
 It follows that
 \[
 \rho_{n,\infty} = \Delta \hat\lambda_{n,\infty}|_\bC \quad\text{and that}\quad \lim_{n\ra\infty} \rho_{n,\infty} = \rho_\infty
 \]
 in the sense of weak convergence of measures.\footnote{We restrict to $\bC$ here out of convenience because the logarithmic singularity of $\hat\lambda_{n,\infty}$ at $z=\infty$ results in a term of the form $-\delta_\infty$ in the Laplacian, which is cancelled by the same term from the Laplacian of the $\log^+\,\abs{z}$ function.} Now, it is easy to check that 
 \[
  \hat\lambda_{n,\infty}(z) = \frac{1}{2^n} \sum_{i=0}^{2^n - 1}\log^+\,\left|2^{i/2^n}z\right| \ra \int_0^1 \log^+\,\abs{2^xz}\,dx\quad\text{as}\quad n\ra\infty.
 \]
 Of course, for each $r>0$, the function of the form $\log^+\,\abs{z/r}$ is subharmonic and has Laplacian the normalized arc length measure on the circle $\{z\in\bC : \abs{z} = r\}$. It follows that $\rho_\infty$ does not charge any individual circle, and so the Portmanteau theorem for weak convergence of measures tells us that 
 \begin{align}
  \begin{split}
  \rho_\infty(D(0,r)) &= \lim_{n\ra\infty}  \Delta \hat\lambda_{n,\infty}(D(0,r))\\
  &= \lim_{n\ra\infty}  \frac{\# \{0\leq i < 2^n : 2^{-i/2^n} < r \}}{2^n},
  \end{split}
 \end{align}
which gives us the desired result. 

 At the prime $p=2$, the result is similar. We claim that 
 \[
  \cJ_{S,2} = [\zeta_{0,1}, \zeta_{0,2}],
 \]
 where $\zeta_{a,b}$ is the point on the Berkovich projective line $\sP^1_2 = \sP^1(\bC_2)$ which, as a seminorm on $\bC_2[x]$, corresponds to the sup norm on the disc $D(a,b)$ in $\bP^1(\bC_2)$, and $[\zeta, \eta]$ denotes the natural line segment on the Berkovich tree connecting two points such that $\zeta\prec \eta$ under the natural ordering on the Berkovich tree. Let $dx$ denote the path length metric on $[\zeta_{0,0}, \zeta_{0,\infty}]$ (the `backbone' of the Berkovich tree for $\sP^1_2$). Then, extending the function $\log^+\abs{2^{i/2^n}z}_2$ to $\sP^1_2$ naturally by continuity, one has the Laplacian
 \[
  \Delta \log^+\abs{2^{i/2^n}z}_2 = \delta_{\zeta_{0, 2^{i/2^n}}} - \delta_\infty,
 \]
 and so the same argument as above in the archimedean case establishes that 
 \[
  d\rho_{S,2}(x) = \begin{cases}
   \displaystyle \frac{dx}{x\log 2} &\text{if }1\leq x\leq 2,\\
   0 &\text{otherwise}.
  \end{cases}
 \]
 Next, we observe that the exceptional set of this system is $E_S=\{0,\infty\}$, as both points are exceptional and fixed for each map in $S$, so they are exceptional for the whole system. 

 Finally, we observe that by Theorem \ref{thm:stoch-equi}, for any $\al\in \bP^1(\Qbar)\setminus\{0,\infty\}$, the random backwards orbit measures $\Delta_{\al,n}$ converge in the weak sense of measures to $\rho_{S,\infty}$ at the infinite place, $\rho_{S,2}$ at $p=2$, and to the standard measure $\rho_{S,p}=\lambda_p$ for $p\neq 2,\infty$.
\end{example}

In the sequel \cite{DFT_II} to this paper the authors will prove various analogues of classical results about the Julia and Fatou sets for stochastic dynamics, including backwards invariance of the Julia set (and hence forward invariance of the Fatou set) and various topological characterizations of the Julia set which can be proven using the equidistribution results of this article. We note that there is a large existing literature of dynamics for semigroups of rational functions, originating in the work of \cite{HinkMartinI}. Note that our stochastic dynamical systems are in fact semigroups of rational functions, but with the additional data of a probability measure on the generating set. However, the theory is closely related, and in fact, we will show that support of our canonical measure coincides with the definition of the Julia set of the semigroup; in particular, so long as the measure $\v_1$ remains strictly positive (i.e., no map is assigned probability 0, and thus effectively removed from the system) then the support of the Julia set is independent of the particular choice of probability $\v_1$. The particular distribution on the Julia set given by the canonical measure, however, still depends on $\v_1$.

\nocite{FRLcorrigendum}

\bibliographystyle{amsalpha}
\bibliography{bib}

\end{document}